\documentclass[12pt,reqno]{amsart}
\usepackage{hyperref}

\usepackage[headsep=30pt,footskip=30pt,twoside=false,bottom=2.5cm]{geometry}

\usepackage[all]{xy}

\xymatrixcolsep{2cm}

\numberwithin{equation}{section}

\theoremstyle{plain}

\newtheorem{theorem}{Theorem}[section]

\newtheorem{proposition}[theorem]{Proposition}

\newtheorem{question}[theorem]{Question}

\newtheorem{lemma}[theorem]{Lemma}

\newtheorem{example}[theorem]{Example}

\newtheorem{definition}[theorem]{Definition}

\theoremstyle{remark}

\newtheorem{remark}[theorem]{Remark}

\newcommand{\Ker}[1]{\ensuremath{\mathrm{Ker} (#1)}}

\newcommand{\Pic}[1]{\ensuremath{\mathrm{Pic} (#1)}}

\newcommand{\ad}[1]{\ensuremath{\mathrm{ad}  (#1)}}
\newcommand{\ENd}[1]{\ensuremath{\mathrm{End}  (#1)}}
\newcommand{\Img}[1]{\ensuremath{\mathrm{Im} (#1)}}
\newcommand{\cat}[1]{\ensuremath{\mathcal{#1}}}
\newcommand{\End}[2][]{\ensuremath{\mathrm{End}_{#1} (#2)}}
\newcommand{\at}[2][]{\ensuremath{\mathrm{at}_{#1} (#2)}}

\newcommand{\END}[2][]{\ensuremath{\mathcal{E}\mathit{nd}_{#1} (#2)}}
\newcommand{\id}[1]{\ensuremath{\mathbf{1}_{#1}}}

\newcommand{\rk}[2][]{\ensuremath{\mathrm{rk}_{#1}(#2)}}

\newcommand{\Z}{\ensuremath{\mathbb{Z}}}

\newcommand{\A}{\ensuremath{\mathbb{A}}}

\newcommand{\p}{\ensuremath{\mathbf{P}}}

\newcommand{\C}{\ensuremath{\mathbb{C}}}

\newcommand{\tr}[1]{\ensuremath{\mathrm{tr}(#1)}}

\newcommand{\struct}[1]{\ensuremath{\mathcal{O}_{#1}}}

\newcommand{\HOM}[3][]{%
  \ensuremath{\mathcal{H}\mathit{om}_{#1}(#2, #3)}}
  
\newcommand{\coh}[3]{\ensuremath{\mathrm{H}^{#1}(#2,\,#3)}}

\begin{document}

\title[Moduli space of Lie algebroid connections]{Line bundles on the moduli space of Lie algebroid connections over a curve }

\author[I. Biswas]{Indranil Biswas}

\address{Department of Mathematics, Shiv Nadar University, NH91, Tehsil Dadri,
Greater Noida, Uttar Pradesh 201314, India}

\email{indranil.biswas@snu.edu.in, indranil29@gmail.com}

\author{Anoop Singh}

\address{Department of Mathematical Sciences, Indian Institute of Technology (BHU), Varanasi, 221005, India} 
\email{anoopsingh.mat@iitbhu.ac.in}

\subjclass[2010]{14D20, 14C22, 14H05, 14M20}
  \keywords{Lie algebroid connection, moduli space, Higgs bundle, Picard group, regular function}

\begin{abstract}
We explore algebro-geometric properties of the moduli space of holomorphic Lie algebroid ($\cat{L }$) connections on a compact Riemann 
surface $X$ of genus $g \,\geq\, 3$. A smooth compactification of the moduli space of $\mathcal{L}$-connections, such that 
underlying vector bundle is stable, is constructed; the complement of the moduli space
inside the compactification is a divisor. A criterion for the numerical effectiveness of the boundary divisor is given.
We compute the Picard group of the moduli space, and analyze Lie algebroid Atiyah bundles 
associated with an ample line bundle. This enables us to conclude that regular functions on the space of certain Lie algebroid 
connections are constants. Moreover, under some condition, it is shown that the moduli space of $\mathcal{L}$-connections does not admit 
non-constant algebraic functions. Rationally connectedness of the moduli spaces is explored.
\end{abstract}

\maketitle

\tableofcontents

\section{Introduction and main results}
\label{Intro}

The notion of Lie algebroid was introduced by J. Pradines \cite{P1, P2} to study the differential groupoids. On the other hand, it 
naturally arises from the properties of space of vector fields on a smooth manifold. The Lie algebroids have been studied in different 
categories namely smooth ($\cat{C}^\infty$), holomorphic as well as algebraic. The papers \cite{KM} \cite{BOM}, \cite{CS}, 
\cite{ELW}, \cite{WE}, \cite{RF} etc. (of course there are many more) have developed the theory of Lie algebroids and Lie algebroid 
connections over manifolds, and varieties from different perspectives. The notion of relative holomorphic Lie algebroid connections was 
introduced in \cite{FB}.

The moduli space of Lie algebroid connections has received great attention over the years because this is a natural generalization of the 
moduli space of holomorphic, logarithmic and meromorphic connections as well as of decorated vector bundles (see Remark \ref{rem:2}).

The Picard group of the moduli space of stable vector bundles over a smooth projective curve has been studied in \cite{DN}, \cite{R}. 
Additionally, the Picard group and regular functions have been investigated for the moduli space of logarithmic connections singular over 
a finite subset of a compact Riemann surface (\cite{BR}, \cite{AS1}, and \cite{AS2}).

Algebro-geometric invariants play a crucial role when studying the properties of any  moduli space.  
In the present article, our aim is to compute algebro-geometric invariants, such as the Picard group, regular function, rational connectedness, of the moduli space of Lie algebroid connections over a compact Riemann surface. We also provide a natural 
compactification of the moduli space of Lie algebroid connections whose underlying vector bundle is stable.

In \cite{S1, S2}, Simpson introduced the notion of $\Lambda$-modules over a smooth projective variety and constructed a moduli space of 
semistable $\Lambda$-modules. In \cite{T1, T2}, Tortella proved an equivalence between the category of split almost polynomial sheaves of 
rings of differential operators and the category of Lie algebroids. Such an equivalence induces a correspondence between the category of 
integrable Lie algebroid connections and the category of modules over the split almost polynomial sheaves of rings of differential 
operators associated with the Lie algebroid; this equivalence preserves the semistability condition, and hence one has the moduli space of Lie 
algebroid connections. Moreover, Krizka \cite{LK1, LK2} used analytic tools to construct moduli space of Lie algebroid connections. In 
\cite{AO}, authors have studied the motives of the Hodge moduli space of Lie algebroid $\lambda$-connections, where $\lambda \,\in\, \C$.

Let $X$ be a compact Riemann surface of genus $g \,\geq\, 2$. We fix a holomorphic Lie algebroid (see Definition \ref{def:1} and Example 
\ref{exmp:1}) $\cat{L} \,=\, (L,\, [.,.],\, \sharp)$ of rank one over $X$ such that $\deg{L^*} \,>\, 2 g -2$, where $L^*$ is the dual of the line 
bundle $L$.
 
The contents of the paper are as follows. 

Section \ref{pre} recalls the notion of holomorphic $\cat{L}$-connection on a holomorphic vector bundle $E$ over $X$ (see Definition 
\ref{def:3}).
Section \ref{Mod} discusses some basics of the moduli space of $\cat{L}$-connections. Throughout this article, we fix integers $r 
\,\geq\, 2$ and $d$ such that $r$ and $d$ are coprime. Let $\cat{M}_{\cat{L}}(r,d)$ denote the moduli space of stable $\cat{L}$-connections 
over $X$ of rank $r$ and degree $d$. Then $\cat{M}_{\cat{L}}(r,d)$ is an irreducible smooth quasi-projective variety of dimension $1 + 
r^2 \deg{L^*}$ (see Theorem \ref{thm:1}).

Section \ref{Comp} considers the moduli space 
$$\cat{M}'_{\cat{L}}(r,d) \,\,\subset\,\, 
\cat{M}_{\cat{L}}(r,d),$$ which is a
 Zariski open dense subset  of $\cat{M}_{\cat{L}}(r,d)$
consisting of those $\cat{L}$-connections for which the underlying vector bundle is stable. Let $\cat{U}(r,d)$ be the moduli space of stable vector bundles of rank $r$ and degree $d$. Then, $\cat{U}(r,d)$
is a smooth projective variety of dimension $r^2(g-1)+1$. 
There is a natural projection 
\begin{equation}
\label{eq:0.1}
p \,\,:\,\, \cat{M}'_{\cat{L}}(r,d) \,\,\longrightarrow\,\, \cat{U}(r,d)
\end{equation}
defined by the forgetful map $(E,\, \nabla_\cat{L}) \,\longmapsto\, E$. We recall the notion of torsor (see Definition \ref{def:4});
it is used in showing that there exists a smooth compactification of the moduli space $\cat{M}'_{\cat{L}}(r,d)$ (see Theorem \ref{thm:1.1}).
Using this smooth compactification of $\cat{M}'_{\cat{L}}(r,d)$,
the Picard group of the moduli space $\cat{M}_{\cat{L}}(r,d)$ is computed . More precisely,
the following is proved (see Theorem \ref{thm:1.2}):
$$\Pic{\cat{M}_{\cat{L}}(r,d)} \,\,\cong\,\, \Pic{\cat{U}(r,d)}.$$

Section \ref{algebraic} investigates the moduli space of $\cat{L}$-connections with fixed determinant. 
Fix a holomorphic line bundle $\xi$ over $X$ with degree $d$ and also fix a holomorphic $\cat{L}$-connection
 $\nabla^\xi_{\cat{L}}$ on $\xi$. The moduli space $\cat{M}_{\cat{L}}(r,\xi)$ that is studied
is described in \eqref{eq:a38}. Consider the moduli space $\cat{M}'_{\cat{L}}(r,\xi) \,\subset\, \cat{M}_{\cat{L}}(r,\xi)$
parametrizing those $\cat{L}$-connections whose underlying vector bundle is stable. The following
is proved (see Proposition \ref{prop:6}):
$$ \Pic{\cat{M}_{\cat{L}}(r, \xi)} \,\,\cong \,\, \Pic{\cat{M}'_\cat{L}(r,\xi)}\cong \Pic{\cat{U}(r,\xi)} \cong \Z,$$
where $\cat{U}(r,\xi)$ denotes the moduli space stable vector bundles of rank $r$ with fixed determinant $\xi$.
 
Let $\cat{P}_\xi $ (see \eqref{eq:a42}) denote the moduli space parametrizing the $L^*$--twisted Higgs bundles $(E,\, \phi)$ on $X$
of rank $r$  such that $E$ is stable and $\bigwedge^rE \,\cong\, \xi$ with $\tr{\phi} \,=\, 0$. Then, $\cat{P}_\xi$ is a vector
bundle over $\cat{U}(r,\xi)$. The dual vector bundle $\cat{P}^*_\xi$ is a Lie algebroid over 
$\cat{U}(r,\xi)$; this construction of Lie algebroid structure on  $\cat{P}^*_\xi$ has been done in \cite[section 4.6]{FB1}.
Let $\Theta$ be the generator of the ample line bundles over $\cat{U}(r,\xi)$. Consider the space 
$\cat{C}onn_{\cat{P}_\xi^*}(\Theta)$ of all  $\cat{P}_\xi^*$-connection on $\Theta$.

Under the assumption that $g\, \geq\, 3$, we compute the space of all global functions on
$\cat{C}onn_{\cat{P}_\xi^*}(\Theta)$ (see Theorem \ref{thm1} and Remark \ref{rem1}).

In Section \ref{Div} we show that the moduli space $\cat{M}_{\cat{L}}(r,d)$ is not rational (see Proposition \ref{thm:6.3}). It
is also shown that $\cat{M}_{\cat{L}}(r,\xi)$ is rationally connected (see Proposition \ref{prop:6.5}).

Let $\p (\cat{F}_\xi)$ be the smooth compactification of the moduli space $\cat{M}'_{\cat{L}}(r, \xi)$ as
in Proposition \ref{prop:5}, and let $${\bf H_0} \,:=\, \p (\cat{F}_\xi) \setminus \cat{M}'_{\cat{L}}(r,\xi)$$
be the smooth divisor. Then, we give a criterion for the numerical effectiveness of the divisor ${\bf H_0}$
(see Proposition \ref{prop:1.5}).
A similar result (see Proposition \ref{prop:8}) can be shown for the divisor $${\bf H} \,\,:=\,\, \p (\cat{F}) \setminus
\cat{M}'_{\cat{L}}(r,d),$$ where $\p (\cat{F})$ is the smooth compactification of the moduli space $\cat{M}'_{\cat{L}}(r,d)$
as in Theorem \ref{thm:1.1}.

\section{Preliminaries}\label{pre}

Let $X$ be a compact connected Riemann surface of genus $g \,\geq\, 2$. Let $T_X$  and $\Omega^1_X$ respectively  denote the holomorphic tangent and cotangent bundle on $X$. 

\begin{definition}\label{def:1}
A {\bf (holomorphic) Lie algebroid} over $X$ is a triple
$\cat{L}\,=\, (L,\, [\cdot,\, \cdot],\, \sharp) $ consisting of
\begin{enumerate}
\item a holomorphic vector bundle $L$ on $X$,
\item a $\C$-bilinear  and skew-symmetric map 
$$[\cdot, \, \cdot] \,: \,L \otimes_{\C} L \,\longrightarrow\, L,$$ called the Lie bracket,
\item a vector bundle homomorphism $$\sharp \,:\, L \,\longrightarrow\, T_X,$$ called the {\bf anchor map} that
induces a homomorphism of Lie algebras from
the sheaf of sections of $L$ to the sheaf of sections of $T_X$ satisfying the following conditions:
\end{enumerate}
\begin{enumerate}
\item[a.] $[u,\, [v,\, w]] + [v,\, [w,\, u]] + [w,\, [u,\, v]] \,=\, 0$ (Jacobi identity),
\item[b.] $[u,\, fv] \,=\, f[u,\, v] + \sharp(u)(f)v $ (Leibniz identity),
\end{enumerate}
for every local holomorphic sections $u, v, w$ of $L$ and every local holomorphic function $f$ on $X$.
\end{definition}

{}From Definition \ref{def:1} it is clear that the space of sections of $L$ has a Lie algebra structure. The rank $\rk{\cat{L}}$ and degree 
$\deg{\cat{L}}$ of a Lie algebroid $\cat{L}$ are by definition the rank and the degree of the underlying vector bundle $L$ respectively.

By a Lie algebroid we will always mean holomorphic Lie algebroid.

\begin{example}\label{exmp:1}
\mbox{} Followings are the well-known examples of Lie algebroids.
\begin{enumerate}
\item \label{a} {\bf Tangent Lie algebroid:}\, The holomorphic tangent bundle 
$T_X$ is a holomorphic Lie algebroid if we take the 
Lie bracket to be the usual Lie bracket operation defined for holomorphic vector fields and take the
anchor map $\sharp$ to be the identity map $\id{T_X}$. We denote it by $\cat{T}_X
\,=\, (T_X,\, [\cdot,\, \cdot],\, \id{T_X})$.

\item \label{b}{\bf Log Lie algebroid:}\, Let $S\,=\,\{x_1,\, \cdots, \,x_m\}$ be a finite subset of $X$;
the corresponding reduced effective divisor $x_1 + \cdots + x_m$ is also denoted
by $S$. Let $T_X(- \log S) \,:=\, T_X \otimes \struct{X}(-S) \subset T_X $ be the line bundle over $X$ consisting of those vector
fields which vanish on $S$. Note $T_X(-\log S)$ is closed under the Lie bracket operation of vector fields.
Take the anchor map to be the inclusion map
$$\sharp \,\,: \,\,T_X(- \log S) \,\,\hookrightarrow\,\, T_X.$$
Then $T_X(-\log S)$ is a Lie algebroid which is called the Logarithmic or Log Lie algebroid, and it is denoted by 
$\cat{T}_X(- \log S) \,:=\, (T_{X}(-\log S),\, [\cdot, \,\cdot], \,\sharp)$.

\item \label{c} {\bf Trivial Lie algebroid:}\, A holomorphic vector 
bundle can be equipped with a Lie algebroid structure
by considering the bracket and anchor maps to be the zero homomorphisms.
Such a Lie algebroid is called the trivial Lie algebroid.   
It is denoted by $\cat{L}_0 \,=\, (L,\, 0,\, 0)$, where $L$ is a holomorphic vector bundle over $X$.
\end{enumerate}
\end{example}

We can define a morphism between Lie algebroids that are not over the same compact Riemann surface.  There are actually two different 
definitions (see, e.g., \cite{KM}). We will mostly be dealing with the Lie algebroid over a single compact Riemann surface $X$. In that 
case, the two definitions coincide.

\begin{definition}
\label{def:2}
Let $\cat{L} \,:=\, (L,\, [\cdot,\, \cdot]_L,\, \sharp_L)$, and 
$\cat{L}' \,:=\, (L',\, [\cdot,\, \cdot]_{L'},\, \sharp_{L'})$ be two Lie algebroids over $X$. A {\bf Lie algebroid morphism}
$\Phi \,:\, \cat{L} \,\longrightarrow\, \cat{L}'$ is a vector bundle morphism
 $\Phi \,:\, L \,\longrightarrow\, L'$ which is a $\C$-Lie algebra homomorphism
satisfying the condition that $\sharp_{L'} \circ \Phi \,=\, \sharp_L$.

We say that $\cat{L}$ and $\cat{L}'$ are isomorphic if the vector bundle morphism $\Phi$ is an isomorphism.
\end{definition}

For a Lie algebroid $\cat{L} \,=\, (L,\, [\cdot,\, \cdot],\, \sharp)$, the morphism $\sharp \,:\, L \,\longrightarrow\, T_X$ is a 
Lie algebroid morphism between $\cat{L}$ and the tangent Lie algebroid $\cat{T}_X$.

\begin{remark}
In \cite[Theorem 3.18]{AO}, authors have classified the Lie algebroid structures on line bundles over $X$.
Also, in \cite[Corollary 3.19]{AO}, it has been shown that the tangent bundle $T_X$ admits only two non-isomorphic
Lie algebroid structures: the trivial one $(T_X, \,0,\, 0)$ and the canonical one given by the Lie bracket of vector fields
$\cat{T}_X \,=\, (T_X,\, [\cdot ,\, \cdot]_{\text{Lie}},\, \sharp \,=\, \id{T_X})$.
\end{remark}

Let $\sharp^* \,: \, \Omega^1_X \,\longrightarrow\, L^*$ be the dual of the anchor map, and let 
$d_{\cat{L}} \,:\, \struct{X} \,\longrightarrow \,L^*$ be the $\C$-derivation defined by 
\begin{equation}
\label{eq:1}
d_{\cat{L}} \,= \, \sharp^* \circ \text{d},
\end{equation}
where 
$\text{d}\,:\, \struct{X} \,\longrightarrow\, \Omega^1_X$ is the de Rham derivation.

\begin{definition}
\label{def:3}
Let $\cat{L} \,=\, (L,\, [\cdot,\, \cdot], \,\sharp)$ be a Lie algebroid over $X$, and let $E$ be a holomorphic vector bundle over
$X$. An $\cat{L}$-connection on $E$ is a $\C$-linear map between the sheaves of holomorphic sections
$$\nabla_{\cat{L}} \,:\, E \,\longrightarrow\, E \otimes L^*$$
such that 
$$\nabla_{\cat{L}} (f s) \,=\, f \nabla_{\cat{L}} (s) + s \otimes d_{\cat{L}}(f),$$
for all locally defined sections $s$ of $E$ and $f$ of 
$\struct{X}$, where $d_{\cat{L}}$ is defined in \eqref{eq:1}. The rank and degree of an $\cat{L}$-connection
$(E, \, \nabla_\cat{L})$ are respectively the rank and degree of the underlying vector bundle $E$.
\end{definition}

Using the anchor map $\sharp$, the exterior derivation $d_{\cat{L}}$ can be extended to higher order 
exterior powers of $L^*$ such that the composition of 
successive derivations vanishes (for more details see \cite{FB}).
Extend $\C$-derivation $d_\cat{L} \,:\, \struct{X} \,\longrightarrow\, L^*$
as defined in \eqref{eq:1} to an operator
$$d^p_\cat{L} \,:\, \bigwedge^p L^* \,\longrightarrow\, \bigwedge^{p+1} L^*$$
by setting 
$$d^p_\cat{L}(\alpha)(a_1, \,\cdots,\, a_{p+1}) \,=\, \sum_{i =1}^{p+1} (-i)^{i+1} \sharp a_i \alpha(a_1,\, \cdots, \,\widehat{a_i},\,
\cdots,\, a_{p+1})$$
$$+ \sum_{i < j} (-1)^{i+j} \alpha ([a_i, \,a_j], a_1, \,\cdots,\, \widehat{a}_i, \,\cdots,\, \widehat{a}_j, \,\cdots,\, a_{p+1}),$$
where $\alpha$ is a section of $\bigwedge^p L^*$, 
$a_i$'s are sections of $L$ for $i \,=\, 1,\, \cdots,\, p+1$,
and $[\cdot,\, \cdot] \,:\, L \otimes_{\C} L\, \longrightarrow\, L$ is the Lie bracket of the Lie algebroid $\cat{L}$. Now, it is
easy to check that $$d^{p+1}_\cat{L} \circ d^p_\cat{L} \,=\, 0,$$
for every $p \,\geq\, 0$, where $d^0_\cat{L} \,=\, d_{\cat{L}}$.
Thus, we get a complex $(\bigwedge^\bullet L^*,\, d^\bullet_\cat{L})$ which is known as the {\bf Chevalley--Eilenberg--de Rham}
complex associated with $\cat{L}$. 

Let $\nabla_{\cat{L}}$ be an $\cat{L}$-connection on $E$. As in the classical case, we shall extend $\nabla_\cat{L}$ to an operator
between the sheaves of holomorphic sections
$$\nabla_{\cat{L}} \,:\, E \otimes \bigwedge^p L^* \,\longrightarrow\, E \otimes \bigwedge^{p+1} L^*$$ 
by setting 
$$\nabla_{\cat{L}}(s \otimes \alpha) \,=\, (\nabla_\cat{L}s) \wedge \alpha + s \otimes d_\cat{L}^p(\alpha)$$
for all locally defined holomorphic sections $s$ of $E$ and $\alpha$ of $\bigwedge^p L^*$.

Now define the $\cat{L}$-curvature $R_{\cat{L}}$ of $\nabla_\cat{L}$ as follows
$$
R_{\cat{L}} \,\,:=\,\, \nabla_\cat{L} \circ \nabla_\cat{L} \,\,:\,\, E \,\,\longrightarrow\,\, E \otimes \bigwedge^2 L^*.$$
As in the classical case, $R_{\cat{L}}$ gives a global holomorphic section of $\END{E}\otimes \bigwedge^2L^*$ because it
is an $\struct{X}$-linear map.

An $\cat{L}$-connection $\nabla_\cat{L}$ is called {\it flat} or {\it integrable}
if the $\cat{L}$-curvature $R_{\cat{L}}$ vanishes identically.
The $\cat{L}$-curvature $R_\cat{L}$ satisfies an analogue of the  classical Bianchi identity.

\begin{example}
\label{exmp:2} We consider the Lie algebroid connections  associated to the Lie algebroids discussed in \text{Example} \eqref{exmp:1}. Let $E$ be a holomorphic vector bundle over $X$.
\begin{enumerate}
\item {\bf Flat holomorphic connection:}\, Consider the tangent Lie algebroid $\cat{T}_X$ as in Example \eqref{exmp:1}\eqref{a}. Then,
a $\cat{T}_X$-connection on $E$  is the usual flat holomorphic connection on $E$. The
flatness of the connection follows from the fact that 
$X$ is a Riemann surface which implies that $\Omega^i_X \,=\, 0$ for $i \,\geq\, 2$.

\item{\bf Flat logarithmic connection:}\, For the Lie algebroid $\cat{T}_X(- \log S)$ in Example 
\eqref{exmp:1}\eqref{b}, a $\cat{T}_X(- \log S)$--connection on $E$ is a flat logarithmic connection on $E$ singular over $S$ (see 
\cite{BM}, \cite{D}).

\item{\bf Twisted Higgs bundle:}\, Consider the trivial Lie algebroid $\cat{L}_0 \,=\, (L,\, 0,\, 0)$ as in Example
\eqref{exmp:1}\eqref{c}. Then a flat $\cat{L}_0$--connection on $E$ is a pair $(E,\, \nabla_{\cat{L}_0})$, where
$\nabla_{\cat{L}_0} \,:\, E \,\longrightarrow\, E \otimes L^*$ is an $\struct{X}$--linear morphism, which is nothing but the
$L^*$--twisted Higgs field on $E$. 
\end{enumerate}
\end{example}

As in \cite{A}, we now describe the first order Lie algebroid jet bundle and the corresponding Atiyah exact sequence.
Consider
$$J^1_{\cat{L}}(E) \,\,:=\,\, E \oplus (E \otimes {L}^*)$$
as a $\C$-module. We equip  $J^1_{\cat{L}}(E)$ with an $\struct{X}$--module structure as follows
$$f \cdot (s,\, \sigma)\,\, =\,\, (f s,\, f \sigma + s \otimes d_{\cat{L}}f ),$$
where $f$, $s$ and $\sigma$ are the local sections of $\struct{X}$, $E$ and $E \otimes \cat{L}^*$ respectively.
This $J^1_{\cat{L}}(E)$ is called the first order Lie algebroid jet bundle associated with $E$.

As in the usual case (see \cite{A}), we have the  $\cat{L}$-Atiyah sequence  associated with $E$ 
\begin{equation}\label{eq:4}
0 \,\longrightarrow\, E \otimes L^* \,\longrightarrow\, J^1_{\cat{L}}(E)\, \xrightarrow{\,\,\,p_E\,\,\,} E
\,\longrightarrow\, 0,
\end{equation}
where $p_E$ is the natural projection. The short exact sequence in \eqref{eq:4} need not split as an $\struct{X}$-module.
Let $$\at[\cat{L}]{E} \,\,\in\,\, \coh{1}{X}{\,\END{E}\otimes L^*}$$ be the extension class of the short exact sequence \eqref{eq:4},
which is known as \cat{L}-{\it Atiyah class} of $E$. A result similar to the following proposition has been proven in \cite[Proposition 17]{T1}.

\begin{proposition}
\label{prop:2.6}
 Let $E$ be a holomorphic vector bundle over $X$ and $\cat{L}$ a Lie algebroid over $X$. Then, the followings are equivalent:
\begin{enumerate}
\item $E$ admits an $\cat{L}$-connection.
\item The $\cat{L}$-Atiyah sequence \eqref{eq:4} associated with $E$ splits holomorphically.
\item The extension class $\at[\cat{L}]{E} \,\in \,\coh{1}{X}{\,\END{E}\otimes L^*}$ vanishes.
\end{enumerate}
\end{proposition}

\begin{proof} $(1) \iff (2)$:\, \, $E$ admits an $\cat{L}$-connection $D\,:\, E
 \,\longrightarrow\, E \otimes L^* $ if and only if the morphism $\phi \,:\, E \,\longrightarrow\, J^1_{\cat{L}}(E)$  defined by
$\phi(s) \,=\, (s,\, D(s))$ is $\struct{X}$-linear, where $s$ is a local section of $E$. 

$(2) \iff (3)$: It is a general fact that the short exact sequence splits if and only if the extension class vanishes.
\end{proof}

Now, we give a sufficient condition for the existence of Lie algebroid connection on a semistable vector bundle over $X$. For any vector 
bundle $E$ over $X$, the slope $\mu(E)$ of $E$ is defined by $$\mu (E) \,=\, \frac{\deg{E}}{\rk{E}}.$$

A vector bundle $E$ over $X$ is said to be semistable if for every non-zero proper subbundle $F$ of $E$, the
inequality $\mu (F) \,\leq\, \mu (E)$ holds.

The following proposition was proved in \cite{AO} (see \cite[Corollary 3.14]{AO}).

\begin{proposition}
\label{prop:1}
Let $E$ be a semistable vector bundle over $X$. Let $\cat{L} \,=\, (L,\, [\cdot,\, \cdot], \,\sharp)$ be a
Lie algebroid such that the vector bundle $L$ is semistable and  $\mu (L^*) \,>\, 2g-2$, where $L^*$ denotes the dual of $L$. Then,
$E$ admits an $\cat{L}$-connection.
Moreover, if $\rk{\cat{L}} \,=\, 1$, then $E$ admits an integrable $\cat{L}$--connection.
\end{proposition}

\begin{proof}
Under the conditions on $E$ and $L$, it is easy to see that $$\mu(\End{E}\otimes L \otimes \Omega^1_X) \,<\, 0,$$
and therefore
$$\coh{0}{X}{\,\End{E}\otimes L \otimes \Omega^1_X} \,= \,0. $$
From Serre duality,
$$\at[\cat{L}]{E} \,\in\, \coh{1}{X}{\,\End{E} \otimes L^*}
\,=\, \coh{0}{X}{\,\End{E}\otimes L \otimes \Omega^1_X}^* \,=\, 0 .$$
Moreover, if $\rk{\cat{L}} \,=\, 1$, we have $\bigwedge^2 L^* =\, \,0$, and hence any $\cat{L}$-connection on $E$ is 
integrable.
\end{proof}

\section{Moduli space of $\cat{L}$-connections}\label{Mod}

In this section, we describe the moduli space of Lie 
algebroid connections over $X$. Henceforth, we assume that the rank of Lie algebroid
$\cat{L} \,=\, (L,\, [\cdot,\, \cdot], \,\sharp)$ is one, that is, $L$ is a line bundle.

\begin{definition}
An $\cat{L}$-connection $(E,\, \nabla_{\cat{L}})$ is said to semistable (respectively, stable) if for every non-zero proper
subbundle $F$ of $E$, which is invariant under 
$\nabla_{\cat{L}}$, that is, $\nabla_{\cat{L}}(F) \,\subset\, F \otimes L^*$, the inequality
$$\mu(F) \,\leq\, \mu(E)\ \ \,\, (\text{respectively, }\ \, \mu(F)\, < \, \mu(E))$$
holds, where $\mu(E)$ denotes the slope of $E$ defined above.
\end{definition}

A morphism between $\cat{L}$-connections $(E,\, \nabla_{\cat{L}})$ and $(E',\, \nabla'_{\cat{L}})$ is a  morphism of vector bundles
$$\phi \,:\,E \,\longrightarrow\, E'$$
such that the following diagram is commutative:
\begin{equation}
\label{eq:cd2}
\xymatrix@C=4em{
E \ar[r]^{\nabla_{\cat{L}}} \ar[d]^{\phi} & E \otimes L^* \ar[d]^{\phi \otimes \id{L^*}} \\
E' \ar[r]^{\nabla'_{\cat{L}}} & E'\otimes L^* .}
\end{equation}
We say that $(E, \,\nabla_{\cat{L}})$
and $(E',\, \nabla'_{\cat{L}})$ are isomorphic if $\phi$ is an isomorphism.

\begin{lemma}
\label{lem:1}
Let $\cat{L}$ be a Lie algebroid over $X$. Let 
$(E,\, \nabla_\cat{L})$ and $(E',\, \nabla'_\cat{L})$ be semistable $\cat{L}$--connections.
Then the following statements hold:
\begin{enumerate}
\item \label{a1} Suppose $(E,\, \nabla_{\cat{L}})$ and $(E',\, \nabla'_{\cat{L}})$ are stable and $\mu(E) \,=\, \mu(E').$
If $$\phi \,:\, (E, \,\nabla_\cat{L})\,\longrightarrow\,
(E',\,\nabla'_\cat{L})$$ is a non-zero morphism of $\cat{L}$--connections, then $\phi$ is an isomorphism.

\item \label{a2} If $(E,\, \nabla_\cat{L})$ is stable, then the only endomorphism of $(E,\, \nabla_\cat{L})$ are constant scalar
multiplications.
\end{enumerate}
\end{lemma}

\begin{proof}\mbox{}
\begin{enumerate}
\item Note that $\Ker{\phi} \,\subset\, E$ is a $\nabla_{\cat{L}}$--invariant coherent
analytic subsheaf of $E$. Since $(E,\, \nabla_{\cat{L}})$ is stable, we 
have $\mu(\Ker{\phi}) \,<\, \mu(E)$. Since $\phi \,\neq\, 0$, we have $\Img{\phi} \,\neq\, 0$.

Consider the Kernel-Image short exact sequence 
$$0 \,\longrightarrow\, \Ker{\phi} \,\longrightarrow\, E \,\longrightarrow\, \Img{\phi} \,\longrightarrow\, 0.$$
Then, $\mu(E) \,<\, \mu(\Img{\phi})$.
Next, consider the Image-Coimage short exact sequence 
$$ 0 \,\longrightarrow\, \Img{\phi} \,\longrightarrow\, E' \,\longrightarrow\, E' / \Img{\phi} \,\longrightarrow\, 0.$$
Note that $\Img{\Phi}$ is a $\nabla'_{\cat{L}}$--invariant coherent analytic subsheaf of $E'$. Since $(E',\,
\nabla'_{\cat{L}})$ is stable, we have $\mu(\Img{\phi}) \,<\, \mu (E')$. Thus, we get that 
$\mu(E) \,<\, \mu(E')$, which contradicts the assumption that
$\mu(E) = \mu (E')$. Therefore, $\Ker{\phi} = 0$ and 
$\Img{\phi} \,=\, E'$.

\item Let  $\phi \,:\, (E,\, \nabla_{\cat{L}})\,\longrightarrow\, (E, \,\nabla_{\cat{L}})$ be an endomorphism.
Let $x \,\in \,X$, and let $\alpha \,\in\, \C$ be an eigenvalue of the linear map $\phi(x) \,:\, E_x
\,\longrightarrow\, E_x$. Note that since $X$ is a compact connected Riemann surface, it does not admit any non-constant holomorphic
function. Therefore the eigenvalues of $\phi$ and also their multiplicities are independent of $x \,\in \,X$.
Since $\nabla_{\cat{L}}$ is $\C$-linear, 
$\phi - \alpha \id{E}$ is an endomorphism of $(E, \nabla_{\cat{L}})$, that is, 
$$\nabla_{\cat{L}} \circ (\phi - \alpha \id{E}) \,\,=\,\, (\phi - \alpha \id{E}) \otimes \id{L^*} \circ \nabla_{\cat{L}}.$$
As $(E, \nabla_{\cat{L}})$ is stable, from Part \eqref{a1} it follows that $\phi - \alpha \id{E}$ is either a zero morphism or
it is an isomorphism. Given that $\alpha$ is an eigenvalue of $\phi(x)$, the kernel of 
$\phi - \alpha \id{E}$ is non-trivial. Thus, $\phi \,=\, \alpha \id{E}$, where $\alpha \,\in\, \C$.
\end{enumerate}
This completes the proof.
\end{proof}

Fix integers $r \,>\, 0$ and $d \,\in\, \Z$. Let 
\begin{equation}
\label{eq:mod1}
\cat{M}_{\cat{L}}(r,d)
\end{equation}
denote the moduli space parametrizing the isomorphism classes of semi-stable 
$\cat{L}$--connections $(E,\, \nabla_{\cat{L}})$ of rank $r$ and degree $d$.
Then, $\cat{M}_{\cat{L}}(r,d)$ is a quasi-projective variety over $\C$.

The moduli space $\cat{M}_{\cat{L}}(r, d)$ can be empty and singular depending on the choices of rank $r$, degree $d$ and Lie algebroid 
$\cat{L}$. For example, if we take $\cat{L} \,=\, \cat{T}_X$, then a $\cat{T}_X$-connection on $E$ is a holomorphic connection on $E$. 
Consequently, from the {\it Atiyah-Weil criterion} (see \cite{A}, \cite{W}) it follows that $\deg{E}\, =\, 0$. Thus, if we choose
$d \,\neq\, 0$, then 
$\cat{M}_{\cat{T}_X}(r,d) \,= \,\emptyset$. Furthermore, when $d \,=\, 0$ and $r \,\geq \,2$, as there exists a strictly semistable
connection, the moduli space $\cat{M}_{\cat{T}_X}(r, 0)$ acquires singularities.

We impose certain conditions on $r$, $d$ and $\cat{L}$ to ensure the existence of  
a non-empty moduli space that is both smooth and irreducible.
Given the assumption that $\rk{\cat{L}} \,=\, 1$, Proposition \ref{prop:1}  makes it evident that it is necessary to have
$\deg{L^*} \,>\, 2g-2$ in order to ensure the non-emptiness for the moduli space 
$\cat{M}_{\cat{L}}(r,d)$. Therefore,
in what follows, it is assumed that $\deg{L^*} \,>\, 2g-2$.
Further, assuming that $r$ and $d$ are coprime, we arrive at a moduli space $\cat{M}_{\cat{L}}(r,d)$ that is both
smooth and irreducible.

The following is proved in \cite{AO}.
 
\begin{theorem}[{\cite[Lemma 5.13, Theorem 7.2]{AO}}]
\label{thm:1}
Let $X$ be a compact connected Riemann surface of genus $g \,\geq\, 2$. Let $\cat{L} \,=\, (L,\, [\cdot,\, \cdot],\, \sharp)$ be a
Lie algebroid such that $\rk{\cat{L}} \,=\, 1$ and $\deg{L^*} \,>\, 2g-2$. Suppose that the integers $r$ and $d$ are coprime with
$r \,\geq\, 2$. Then the moduli space $\cat{M}_{\cat{L}}(r,d)$
of stable $\cat{L}$-connections is an  irreducible smooth quasi-projective variety of dimension $1+ r^2 \deg{L^*}$. 
\end{theorem}

The following remarks emphasize the significance of adopting the overarching perspective of Lie algebroid connections.

\begin{remark} \label{rem:2} 
\mbox{}
\begin{enumerate}
\item Let $D \,=\, \sum_{i = 1}^m a_i x_i$ be an effective divisor on $X$ with $a_i \,\geq\, 1$ and $x_i \,\in\, X$. Let
$\cat{M}_{conn}(D, r, d)$ be 
the moduli space of rank $r$ and degree $d$ semistable meromorphic connections with poles of order at most $a_i$ over each $x_i \in D$. 
If we take the Lie algebroid $\cat{T}_X(-D) \,:=\, (T_X(-D) \,:=\, T_X \otimes \struct{X}(-D),\, [\cdot,\, \cdot],\, \sharp)$, where
$[\cdot , \, \cdot]$ is the usual Lie bracket operation on the vector fields and $\sharp$ is the inclusion morphism
$\sharp \,:\, T_X(-D)\, \hookrightarrow \,T_X$, then, $\cat{M}_{\cat{T}_X(-D)} (r, d) \,\cong\, \cat{M}_{conn}(D, r, d)$.

\item In \cite[Corollary 3.20]{AO}, it has been shown that for any Lie algebroid $\cat{L} \,= (L,\, [\cdot ,\, \cdot],\, \sharp)$ with 
$\deg (L^*) \,> \,2g-2$ and $\sharp \,\neq\, 0$, we have $\cat{M}_{\cat{L}}(r,d) \,\cong\, \cat{M}_{conn}(D, r, d)$ for a unique 
effective divisor $D$ in the linear system $\vert L^{-1} \otimes T_X \vert$.

\item Thus, the formalism of Lie algebroid connections can be effectively applied to uniformly address a wide array of moduli spaces 
encompassing logarithmic connections, meromorphic connections with poles over distinct divisors. This approach will enable us to uncover 
meaningful interconnections among their geometries.
\end{enumerate}
\end{remark}

Next, let's assume that the Lie algebroid $\cat{L}$ is trivial, that is, $\cat{L} \,=\, (L,\, 0,\, 0)$. Then the
$\cat{L}$--connections correspond to the twisted Higgs bundles with a twist by $L^*$ (as described in the third point of
Example \ref{exmp:2}).

Let $\cat{N}_{L^*}(r,d)$ denote the moduli space of semi-stable $L^*$--twisted Higgs bundles with rank $r$ and degree $d$. This 
space $\cat{N}_{L^*}(r,d)$ was constructed in \cite{N}. Under the conditions stated in Theorem \ref{thm:1}, that is, $\deg{L^*} \,>\, 
2g-2$, $r$ and $d$ coprime, and $r \geq 2$, the moduli space $\cat{N}_{L^*}(r,d)$ is an irreducible smooth quasi-projective variety of 
dimension $1 + r^2 \deg{L^*}$, as established in \cite[Theorem 1.2, Proposition 3.3]{BGL}.

The moduli space $\cat{N}_{L^*}(r,d)$ of $L^*$--twisted Higgs bundles is equipped with the Hitchin map
\begin{equation}\label{eq:5}
\cat{H} \,\,:\,\, \cat{N}_{L^*}(r,d) \,\,\longrightarrow\,\, \cat{A} \,\,:=\,\, \bigoplus_{i =1}^r \coh{0}{X}{(L^*)^i}
\end{equation}
defined by $(E, \phi)\, \longmapsto\,\sum_{i = 1}^{r} \tr{\bigwedge^i \phi}$, where $\phi
\,:\, E \,\longrightarrow\, E \otimes L^* $ is an $\struct{X}$--linear morphism. From
\cite[Theorem 6.1]{N}, the map $\cat{H}$ in \eqref{eq:5} is proper.

Let
\begin{equation}
\label{eq:m2}
\cat{P} \, := \, \cat{P}_{L^*}(r,d) \, \subset \, \cat{N}_{L^*}(r,d)
\end{equation}
be the moduli space of $L^*$--twisted 
Higgs bundles $(E, \,\phi)$ such that the underlying vector bundle $E$ is 
stable. Then from \cite[Theorem 2.8(A)]{M}, $\cat{P}$
is a Zariski open dense subset of $\cat{N}_{L^*}(r,d)$.
We can restrict the map $\cat{H}$ in \eqref{eq:5} to $\cat{P}$; let
\begin{equation}
\label{eq:hit}
\cat{H}_{\cat{P}}\,\, :\,\, \cat{P} \,\,\longrightarrow\,\, \cat{A}.
\end{equation}
denote the restriction map.

Using the similar techniques as in \cite{BNR}, one can construct a $1$-dimensional scheme $X_s$ and a finite morphism $X_s
\,\longrightarrow\, X$, for 
every point $s \,\in\, \cat{A}$. This $X_s$ is called the spectral curve associated to the point $s$; it can be singular.
Consider the set
$$U\, =\, \{s \,\in\, \cat{A}\,\,\big\vert\,\, X_s ~\text{ is~integral~and~ smooth}\}\, \subset\, \cat{A}.$$
Then, under the assumption on $\cat{L}$ (i.e., $\deg{L^*} \,>\, 2g-2$), this $U$ is a non-empty Zariski open subset of $\cat{A}$.

According to \cite[Theorem 2.2.1]{FB1}, for any general point $s \,\in\, \cat{A}$, the fiber $\cat{H}^{-1}(s)$ is an abelian variety, denoted 
as $\text{J}(X_s)$; it is the Jacobian variety of the spectral curve $X_s$. Furthermore, as outlined in \cite[Remark 2.2.2]{FB1}, for any 
general points $s \,\in\, \cat{A}$, if $r \,\geq \,3$ and $g \,\geq\, 2$, the fiber $\cat{H}_{\cat{P}}^{-1}(s)$ is of the form $A_s 
\setminus F_s$, where $A_s$ is an abelian variety and $F_s$ is a closed subset of $A_s$ satisfying $\text{codim}(F_s,\, A_s) \,\geq\, 2$.

Let $\cat{U}(r,d)$ be the moduli space of stable vector bundle of rank $r$ and degree $d$, where $r$ and $d$ are coprime. Then, $\cat{U}(r,d)$ is a smooth projective variety of dimension $r^2(g-1)+1$.
Let 
\begin{equation}
\label{eq:6}
\pi \,\,:\,\, \cat{P} \,\,\longrightarrow\,\, \cat{U}(r,d)
\end{equation}
be the morphism defined by $(E,\, \phi) \,\longmapsto \,E$. In view of \cite[Lemma 1.3.1]{FB1}, $\cat{P}$ is a vector bundle over
$\cat{U}(r,d)$ with fiber $\pi^{-1}(E) \,=\, \coh{0}{X}{\,\End{E} \otimes L^*}$.

\section{Compactification and Picard group} 
\label{Comp}

We continue using notation and assumptions from the previous sections. In particular, henceforth we assume that
\begin{equation}\label{an1}
L\,=\, T_X(- \log S)\,=\, TX\otimes {\mathcal O}_X(\sum_{i=1}^m x_i),
\end{equation}
where $\{x_1,\, \cdots, \,x_m\}\, \subset\, X$ is a finite subset of distinct points, and the anchor map
$\sharp \,:\, L = \,T_X(- \log S)\longrightarrow\, T_X$ coincides with the natural inclusion homomorphism.

Let $$\cat{M}'_{\cat{L}}(r,d)\,\, \subset\,\, \cat{M}_{\cat{L}}(r,d)$$ denote the moduli space  of $\cat{L}$--connections
$(E,\, \nabla_{\cat{L}})$ such that the underlying vector bundle $E$ is stable.  It follows from  \cite[Theorem 2.8(A)]{M}
that $\cat{M}'_{\cat{L}}(r,d)$ is a Zariski open dense subset of $\cat{M}_{\cat{L}}(r,d)$.

Let
\begin{equation*}
\label{eq:7}
p \,\,:\,\, \cat{M}'_{\cat{L}}(r,d) \,\,\longrightarrow\,\, \cat{U}(r,d)
\end{equation*}
be the forgetful map as defined in \eqref{eq:0.1}.

We recall the definition of a torsor, and show that 
the moduli space $\cat{M}'_{\cat{L}}(r,d)$ is a $\cat{P}$--torsor over $\cat{U}(r,d)$.

\begin{definition}
\label{def:4}
Let $M$ be a connected complex algebraic variety. Let $\pi\,:\, 
\cat{V} \,\longrightarrow\, M$, be an algebraic  vector bundle.
A $\cat{V}$--\emph{torsor} on $M$ is an algebraic  fiber bundle $p\,:\, Z \,\longrightarrow\, M$,
and an algebraic map from the fiber product
\begin{equation*}
\label{eq:20.1}
\varphi\,:\, Z \times_M \cat{V} \,\longrightarrow\, Z
\end{equation*} 
such that the following conditions are satisfied:
\begin{enumerate}
\item $p \circ \varphi\, =\, p \circ p_Z$, where $p_Z$ is the 
natural projection of $Z \times_M \cat{V} $ to $Z$;
\item the map $Z \times_M \cat{V} \,\longrightarrow\, Z \times_M  Z$ 
defined by $p_Z \times \varphi$ is an isomorphism;
\item $\varphi(\varphi(z,\,v),\,w) \,=\, \varphi(z, \,v+w)$ for all $z\, \in\, Z$ and $v,\, w\, \in\,
\cat{V}$ such that $\pi(v)\,=\, \pi(w)\,=\, p(z)$.
\end{enumerate}
\end{definition}

\begin{proposition}
\label{prop:2}
The isomorphic classes of $\cat{V}$--torsors over $M$ are 
parametrized by $\coh{1}{M}{\cat{V}}$.
\end{proposition}

\begin{proof}
Let $p\,:\, Z\,\longrightarrow\, M$ be a $\cat{V}$--torsor. Let $\{U_i\}_{i \in I}$  be a covering of $M$ by open sets such
that there are sections 
$$\sigma_i\,:\, U_i \,\longrightarrow\, Z{\vert_{ U_i}} $$
for all $i \,\in\, I$. Since $\cat{V}$ acts on $Z$ freely and transitively on the fibers, it follows that
$\sigma_i - \sigma_j$ is a section of $\cat{V}\vert_{U_i \cap U_j}$. Then $\{\sigma_i - \sigma_j\}_{i,j \in I}$
forms a $1$-cocycle with values in $\cat{V}$, and hence defines a class in $\coh{1}{M}{\,\cat{V}}$. This class
is trivial if and only if $Z$ is isomorphic to $\cat{V}$.
\end{proof}

\begin{proposition}\label{prop:3}
Let $\pi \,:\, \cat{P} \,\longrightarrow\, \cat{U}(r,d)$ be the algebraic vector bundle defined in \eqref{eq:6}. Then, the fiber bundle
$p\,:\,\cat{M}'_{\cat{L}}(r,d) \,\longrightarrow\, \cat{U}(r,d)$ defined in \eqref{eq:0.1} is a $\cat{P}$--torsor over 
$\cat{U}(r,d)$.
\end{proposition}

\begin{proof}
Let $(E,\, \nabla_{\cat{L}})$ and $(E,\, \nabla'_{\cat{L}})$ be two $\cat{L}$--connections on $E$.
Then $$\nabla_{\cat{L}} - \nabla'_{\cat{L}} \,\in\, \coh{0}{X}{\,\End{E}\otimes L^*}.$$
Conversely, given any $\omega \,\in\, \coh{0}{X}{\End{E}\otimes L^*}$, it is evident that $\nabla_{\cat{L}} + \omega$ is again an 
$\cat{L}$--connection on $E$. Thus, $p^{-1}(E) \,\subset\, 
\cat{M}'_{\cat{L}}(r,d)$ is an affine space modelled over the vector space $\coh{0}{X}{\,\ENd{E} \otimes L^*}$. 

Note that the fiber of the bundle $ \pi\,:\, \cat{P} \,\longrightarrow\, \cat{U}(r,d)$ over $E$ is 
$\coh{0}{X}{\,\ENd{E} \otimes L^*}$. Therefore, we get a natural
action of $\pi^{-1}(E)$ on $p^{-1}(E)$, 
that is, a map
\begin{equation*}
\label{eq:20.2}
\pi^{-1}(E) \times p^{-1}(E) \,\longrightarrow\, p^{-1}(E)
\end{equation*}
defined by $(\omega, \,\nabla_{\cat{L}})\, \longmapsto\, \omega + \nabla_{\cat{L}}$.
This action on the fiber is free and transitive.
It induces a morphism on the fiber products
\begin{equation*}
\label{eq:20.3}
\varphi\,:\, \cat{P} \times_{\cat{U}(r,d)} \cat{M}'_{\cat{L}}{(r,d)}\,\longrightarrow\, \cat{M}'_{\cat{L}}{(r,d)},
\end{equation*}
which satisfies the above conditions in the definition of 
the torsor.
\end{proof}

Let $\cat{C}onn_{\cat{L}}(E)$ denote the space of all $\cat{L}$--connections $\nabla_{\cat{L}}$ on $E$ such that $(E,\, \nabla_{\cat{L}})$ 
is stable. Note that $\cat{C}onn_{\cat{L}}(E)$ is an affine space modelled over the vector space $\coh{0}{X}{\End{E}\otimes L^*}$.
 
Given an automorphism $\Phi$ of $E$ and an $\cat{L}$--connection $\nabla_{\cat{L}}$ on $E$, the $\C$--linear morphism 
$\Phi \otimes \id{L^*} \circ \nabla_{\cat{L}} \circ 
\Phi^{-1}$ defines an $\cat{L}$-connection on $E$.
In fact, $$(\nabla_{\cat{L}},\,\Phi)\,\, \longmapsto\,\, \Phi \otimes 
\id{L^*} \circ \nabla_{\cat{L}} \circ \Phi^{-1}$$ 
defines a natural action of $\text{Aut}(E)$ on $\cat{C}onn_{\cat{L}}(E)$, called gauge transformation.
We would like to compute the dimension of the quotient 
space $\cat{C}onn_{\cat{L}}(E) / \text{Aut}(E)$, that 
parametrizes all isomorphic $\cat{L}$-connections on $E$. 

The Lie algebra of the holomorphic automorphism group 
$\text{Aut}(E)$ is $\coh{0}{X}{\End{E}}$. Therefore, 
$$\text{dim}~\text{Aut}(E) = \text{dim}~\coh{0}{X}{\End{E}}.$$

Choose any $\nabla_{\cat{L}}\, \in\, \cat{C}onn_{\cat{L}}(E)$. Then, from Lemma \ref{lem:1} \eqref{a2}, the isotropy subgroup 
$$\text{Aut}(E)_{\nabla_{\cat{L}}} \,= \,\{\Phi \,\in\, \text{Aut}(E)\,\,\big\vert\,\,\Phi \otimes 
\id{L^*} \circ \nabla_{\cat{L}} \circ \Phi^{-1} \,=\, \nabla_{\cat{L}} \}$$ of $\text{Aut}(E)$ is the group of scalar 
automorphisms of $E$.
Then, the dimension of the space $\cat{C}onn_{\cat{L}}(E) / \text{Aut}(E)$ is 
\begin{equation}
\label{eq:8}
\dim\, \coh{0}{X}{\, \End{E}\otimes L^*} - 
\dim\, \coh{0}{X}{\, \End{E}} + 1.
\end{equation}

\begin{lemma}
\label{lem:2}
Let $E$ be a stable vector bundle on $X$ of rank $r$ and degree $d$. Then  $$\cat{C}onn_{\cat{L}}(E) / \text{Aut}(E)$$
is an affine space modelled on the vector space $\coh{0}{X}{\, \End{E}\otimes L^*}$, and dimension of this space is
$r^2(\deg L^* - g+1)$.
\end{lemma}

\begin{proof}
Note that when $E$ is stable, $\coh{0}{X}{\End{E}} \,=\, \C \cdot \id{E}$, that is, automorphisms of $E$ are just non-zero
constant scalar multiplications, and in that case 
the quotient space $\cat{C}onn_{\cat{L}}(E) / \text{Aut}(E)$ is equal to the affine space modelled over the vector
space $\coh{0}{X}{\End{E}\otimes L^*}$. 
This is the case with the fibers of the morphism $p$
defined in \eqref{eq:0.1}, because the underlying vector bundle is stable.

In view of \eqref{eq:8}, it is enough to compute the dimension of 
$ \coh{0}{X}{\End{E}\otimes L^*}$.
Given that $E$ is stable, the vector bundle $\End{E} \otimes \Omega^1_X \otimes L$ is semistable. Since 
$$\mu (\End{E} \otimes \Omega^1_X \otimes L) \,=\, 2g-2 + \deg(L) \,<\, 0,$$ we have
$\coh{0}{X}{\,\End{E} \otimes \Omega^1_X \otimes L} \,=\, 0$. Thus, from Serre duality,
$$\coh{1}{X}{\,\End{E} \otimes L^*} \,=\, \coh{0}{X}{\,\End{E} \otimes \Omega^1_X \otimes L}^* \,= \,0.$$
From Riemann-Roch theorem for $X$, 
$$\dim\, \coh{0}{X}{\, \End{E}\otimes L^*} \, = \, \deg(\End{E} \otimes L^*) \,+ \, \rk{\End{E} \otimes L^*}(1-g)$$
$$
\,= \, r^2(\deg{L^*} -g +1).$$ 
This completes the proof.
\end{proof}
 
\begin{theorem}
\label{thm:1.1}
There exists an algebraic vector bundle 
\begin{equation}
\label{eq:phi}
\Phi \,\,:\,\, \cat{F} \, \,\longrightarrow\, \, \cat{U}(r,d)
\end{equation}
of rank $r^2(\deg{L^*} -g +1) +1$ such that 
$\cat{M}'_{\cat{L}}(r,d)$ is embedded in $\p (\cat{F})$ with
\begin{equation}
\label{eq:h}
{\bf H} := \p (\cat{F}) \setminus \cat{M}'_{\cat{L}}(r,d)
\end{equation}
being an hyperplane at infinity.
\end{theorem}

\begin{proof}
\label{proof_thm_1.1}
Let $\cat{G}$ be an affine bundle modeled on a vector bundle $\cat{E}$ of rank $n$ over $\cat{U}(r,d)$. Now, using the standard inclusion of the affine group in $\text{GL}(n+1, \C)$, we obtain a 
vector bundle $\cat{F}$ of rank $n +1$ along with an 
embedding of $\cat{G}$ in $\p (\cat{F})$ as an open 
subset with complement being a hyperplane.

From Proposition \ref{prop:3}, since $\cat{M}_{\cat{L}}'(r,d)$ is a
$\cat{P}$--torsor over $\cat{U}(r,d)$, the above construction yields  
an algebraic vector bundle $\cat{F}$
over $\cat{U}(r,d)$. In this construction, $\cat{M}'_{\cat{L}}(r,d)$ is embedded 
in $\p (\cat{F})$, and  the complement $\p(\cat{F}) \setminus \cat{M}'_{\cat{L}}(r,d)$ forms an hyperplane 
at infinity which denoted as ${\bf H}$. 

The rank of the vector bundle $\cat{F}$ is $r^2(\deg L^* - g+1) + 1$, a fact that readily follows from Lemma \ref{lem:2}.
\end{proof}

Thus, we get a smooth compactification $\p (\cat{F})$ of 
$\cat{M}'_{\cat{L}}(r,d)$ by a smooth divisor $\textbf{H}$ at infinity.

We have the natural inclusion morphism 
\begin{equation}
\label{eq:i}
\iota\,: \,\cat{M}'_{\cat{L}}(r,d) \,\hookrightarrow \,\cat{M}_{\cat{L}}(r,d)
\end{equation}
that induces a homomorphism on the Picard groups
\begin{equation}
\label{eq:0.10}
\iota^* \,:\, \Pic{\cat{M}_{\cat{L}}(r,d)} \,\longrightarrow\, \Pic{\cat{M}'_{\cat{L}}(r,d)}
\end{equation}
defined by the restriction of line bundles.
Further, the morphism $p$ defined in \eqref{eq:0.1} induces a homomorphism of Picard groups
\begin{equation}
\label{eq:0.11}
p^* \,:\, \Pic{\cat{U}(r,d)} \,\longrightarrow\, \Pic{\cat{M}'_{\cat{L}}(r,d)}
\end{equation}
using the pullback operation of line bundles.

\begin{theorem}\label{thm:1.2}
Let $g \,\geq\, 3$ and $r \,\geq\, 2$. Then the two homomorphisms 
$$\iota^* \,:\, \Pic{\cat{M}_{\cat{L}}(r,d)} \,\longrightarrow\, \Pic{\cat{M}'_{\cat{L}}(r,d)}$$ and 
$$p^* : \Pic{\cat{U}(r,d)} \longrightarrow \Pic{\cat{M}'_{\cat{L}}(r,d)},$$
constructed in \eqref{eq:0.10} and 
\eqref{eq:0.11} respectively, are isomorphisms.
\end{theorem}

\begin{proof}
Let $Z \,:=\, \cat{M}_{\cat{L}}(r,d) \setminus \cat{M}'_{\cat{L}}(r,d)$, that is, $Z$ is the locus of those $\cat{L}$--connections
such that the underlying vector bundle is not stable. 
Then, from \cite[Lemma 7.1]{AO}, we have 
 $$\text{codim}(Z, \,\cat{M}_{\cat{L}}(r, d)) \,\,\geq \,\,(g-1)(r-1).$$
 Thus, for $g \,\geq\, 3$ and $r \,\geq \,2$, we have 
 $\text{codim}(Z, \,\cat{M}_{\cat{L}}(r, d)) \,\geq\, 2,$
 and hence the homomorphism $\iota^*$ in \eqref{eq:0.10} is an isomorphism.

Next, to show that $p^*$  in \eqref{eq:0.11} is an isomorphism, first we show that $p^*$ is injective.
Let $\eta \,\longrightarrow\,
\cat{U}(r, d)$ be an algebraic line bundle such that $p^* \eta$ is a trivial line bundle
over $\cat{M}'_{\cat{L}}(r, d)$. 
Since $p^* \eta$ is a trivial line bundle, we get a  
 nowhere vanishing section of $p^* \eta$ over $\cat{M}'_{\cat{L}}(r, d)$. Now, fix a nowhere vanishing  section  
$ t \,\in\, \coh{0}{\cat{M}'_{\cat{L}}(r, d)}{\, p^* \eta},$
and choose a
point $E \,\in \,\cat{U}(r, d)$.  Then, from the following commutative diagram 
\begin{equation}
\label{eq:a36}
\xymatrix{
p^* \eta \ar[d] \ar[r]^{\widetilde{p}} & \eta \ar[d] \\
\cat{M}'_{\cat{L}}(r, d) \ar[r]^{p} & \cat{U}(r, d)\\
}
\end{equation}
it follows that
\begin{equation*}
t|_{p^{-1}(E)}\,\,: p^{-1}(E) \,\,\longrightarrow\,\, \eta(E)
\end{equation*}
is nowhere vanishing map. Observe that $p^{-1}(E) \,\cong\, \C^N$ and $\eta(E) \,\cong\, \C$, where $N
\,=\, r^2 (\deg L^* - g +1)$ (see Lemma \ref{lem:2}). Further, any nowhere vanishing algebraic function on an affine space
$\C^N$ is a constant function, that is, $t|_{p^{-1}(E)}$ is a constant
function, and therefore, it corresponds to a non-zero vector $\alpha_{E} \,\in\, \eta(E)$.
Since $t$ is constant on each fiber of $p$, the trivialization $t$ of $p^*\eta$ descends to a trivialization of the
line bundle $\eta$ over $\cat{U}(r, d)$, and hence a nowhere vanishing
section of $\eta$ over $\cat{U}(r, d)$ is produced.  Thus,  $\eta$ is a trivial line bundle
over $\cat{U}(r, d)$.
  
It remains to show that $p^*$ is a surjective morphism. Let 
$\vartheta \,\longrightarrow\, \cat{M}'_{\cat{L}}(r,d)$ be an algebraic line 
bundle. Since $\p(\cat{F})$ is a smooth compactification of $\cat{M}_{\cat{L}}'(r,d)$, we can extend $\vartheta$ to a line bundle
$\vartheta'$ over $\p (\cat{F})$.
Now, from the morphism $$\widetilde{\Phi}\,:\, \p(\cat{F})
\,\longrightarrow\, \cat{U}(r,d)$$ induced by the morphism in \eqref{eq:phi}, we have 
\begin{equation}
\label{eq:a31}
 \Pic {\p (\cat{F})} \,\cong\, \widetilde{ \Phi}^*\Pic{\cat{U}(r, d)}\oplus  \Z \struct{\p (\cat{F})}(1).
\end{equation}
Consequently,
\begin{equation}
\label{eq:a32}
\vartheta' \,=\, \widetilde{\Phi}^* \Lambda \otimes \struct{\p (\cat{F})}(m),
\end{equation}
where $\Lambda$ is a line bundle over $\cat{U}(r,d)$ and 
$m \,\in\, \Z$. Since ${\bf H} \,=\, \p (\cat{F}) \setminus \cat{M}'_{\cat{L}}(r, d)$ in \eqref{eq:h} is the 
hyperplane at infinity, using \eqref{eq:a31}, the line bundle 
$\struct{\p (\cat{F})}({\bf H})$ associated to the divisor ${\bf H}$ can be expressed as
\begin{equation}
\label{eq:a33}
\struct{\p (\cat{F})}({\bf H}) \,=\, \widetilde{\Phi}^* \Gamma \otimes \struct{\p (\cat{F})}(1),
\end{equation}
where $\Gamma$ is a line bundle over $\cat{U}(r, d)$.
Now, from \eqref{eq:a32} and \eqref{eq:a33}, we get that
\begin{equation*}
\label{eq:a34}
\vartheta' \,=\, \widetilde{\Phi}^*(\Lambda \otimes 
(\Gamma^{*})^{\otimes m}) \otimes \struct{\p (\cat{F})}(m{\bf H}),
\end{equation*}
where $\Gamma^*$ denotes the dual of $\Gamma$.
Since the restriction of the line bundle $\struct{\p 
(\cat{F})}({\bf H})$ to the compliment $$\p (\cat{F}) \setminus {\bf H} \,=\, 
\cat{M}'_{\cat{L}}(r, d)$$ is trivial, and the restriction of $\widetilde{\Phi}$ to $\cat{M}'_{\cat{L}}(r, d)$
is the map $p$ defined in \eqref{eq:0.1}, we conclude that
\begin{equation*}
\label{eq:a35}
\vartheta \,\,=\,\,  p^*(\Lambda \otimes (\Gamma^{*})^{\otimes m}).
\end{equation*}
This completes the proof of the theorem.
\end{proof}

\section{the moduli space with fixed determinant}
\label{algebraic}

In this section we explore the moduli space of Lie algebroid connections with a fixed determinant. Let $\xi$ be a holomorphic line bundle 
over $X$ of degree $d$ such that $d$ is coprime to $r$. Fix a $\cat{L}$--connection $$\nabla^\xi_{\cat{L}} \,\,:\,\, \xi
\,\,\longrightarrow\,\, \xi \otimes \cat{L}^*$$ 
on $\xi$. Given an $\cat{L}$--connection $\nabla_{\cat{L}}$ on $E$ of rank $r$, we have an $\cat{L}$--connection $\tr{\nabla_{\cat{L}}}$ on 
$\bigwedge^rE$.

Consider the moduli space
\begin{equation}
\label{eq:a38}
\cat{M}_{\cat{L}}(r, \xi) \,\,\subset\,\, \cat{M}_{\cat{L}}(r,d)
\end{equation}
 parametrizing the isomorphic class of pairs $(E,\, \nabla_{\cat{L}})$ such that 
$$(\bigwedge^rE,\, \tr{\nabla_{\cat{L}}}) \,\,\cong\, (\xi,\, \nabla^\xi_{\cat{L}}).$$
Then, $\cat{M}_{\cat{L}}(r, \xi)$ is a smooth quasi-projective variety of dimension $(r^2-1) \deg{L^*}$ (see \cite[Proposition 9.7]{AO}).

Let 
\begin{equation}
\label{eq:a39}
\cat{M}_{\cat{L}}'(r, \xi) \,\,\subset \,\, \cat{M}_{\cat{L}}(r, \xi)
\end{equation}
be the subset consisting of those $\cat{L}$--connections whose underlying vector bundle is stable.
Then, $\cat{M}_{\cat{L}}'(r, \xi)$ is  a Zariski dense open subvariety of $\cat{M}_{\cat{L}}(r, \xi)$; this follows
immediately from the openness of the stability condition (see \cite{M}).

Let $\cat{U}(r, \xi)$ be the moduli space of stable vector bundles with fixed determinant $\xi$. Then, 
$\cat{U}(r, \xi)$ is a smooth projective variety of dimension $(r^2-1)(g-1)$. We have a natural projection
\begin{equation}
\label{eq:a40}
q \,:\, \cat{M}_{\cat{L}}'(r, \xi) \,\longrightarrow\, \cat{U}(r, \xi)
\end{equation} 
defined by $(E,\, \nabla_{\cat{L}}) \,\longmapsto\, E$.

Next, consider the moduli space $\cat{N}_{L^*}(r,d)$ of 
rank $r$, degree $d$, semi-stable $L^*$--twisted Higgs bundles as described above.
Let 
\begin{equation}
\label{eq:a41}
\cat{N}_{L^*}(r, \xi) \,\,\subset\,\, \cat{N}_{L^*}(r,d)
\end{equation}
be the moduli space parametrizing the isomorphic class of pairs $(E, \,\phi)$ such that $\bigwedge^r E \,\cong\, \xi$
and $\tr{\phi} \,=\, 0$.
Further, let
\begin{equation}
\label{eq:a42}
\cat{P}_{\xi} \,\,:= \,\,\cat{P}_{L^*}(r, \xi) \subset \cat{N}_{L^*}(r, \xi)
\end{equation}
be the moduli space of those $L^*$--twisted Higgs bundles in $ \cat{N}_{L^*}(r, \xi)$ such that underlying vector bundle
is stable. 

Let 
\begin{equation}
\label{eq:a43}
\pi' \,:\, \cat{P}_{\xi} \,\longrightarrow\, \cat{U}(r, \xi)
\end{equation}
be the morphism of varieties defined by $(E, \,\phi) \,\longmapsto\, E$. Then, $\cat{P}_{\xi}$ is a vector bundle
over $\cat{U}(r, \xi)$ whose fiber over $E$ is $\pi'^{-1}(E) \,=\, \coh{0}{X}{\,\ad{E} \otimes L^*}$, where $\ad{E}
\,\subset\, \End{E}$ is a subspace consisting of all those endomorphisms of $E$ 
whose trace is zero.

Restrict the Hitchin map $\cat{H}$ defined in 
\eqref{eq:5} to $\cat{P}_{\xi}$ and denote it by $\cat{H}_\xi$. Thus, we have  
\begin{equation}
\label{eq:a44}
\cat{H}_{\xi} \,:\, \cat{P}_\xi \,\longrightarrow\, \cat{B} \,:=\, \bigoplus_{i =2}^r \coh{0}{X}{(L^*)^i}
\end{equation}
defined by $$(E,\, \phi) \,\longmapsto\, \sum_{i =2}^r \tr{\bigwedge^i \phi}$$ (for more details see \cite{FB1}). 
For any general $b \,\in\, \cat{B}$, let $X_b$ denote the spectral curve 
defined by $b$, which is a ramified $r$-sheeted covering $\epsilon \,:\, X_b \,\longrightarrow\, X $ of $X$.
Recall that, for a general $b \in \cat{B}$, the kernel of the norm  map 
$$\text{Nm}\, :\, \text{J}(X_b) \,\longrightarrow\, \text{J}(X)$$
between the Jacobians is called the Prym variety and is denoted by $\text{Prym}(X_b/X)$.

For a general $b \,\in\, \cat{B}$, the fiber $\cat{H}_{\xi}^{-1}(b)$ is isomorphic to the open subset $A_b$
of $\text{Prym}(X_b/X)$ consisting of isomorphic class of line bundles $M$ over $X_b$ such that the push-forward $\epsilon_*M$ is
a stable vector bundle of degree $d$.
Let $F_b \,:=\, \text{Prym}(X_b/X) \setminus A_b$ denote the complement. Then, from \cite[Proposition 5.7]{BNR},
$$\text{codim}(F_b,\, \text{Prym}(X_b/X)) \,\geq\, 2.$$

Considering again the moduli space $\cat{M}_{\cat{L}}'(r, \xi)$ and using the similar statements as in Proposition \ref{prop:3},  we can show the following.

\begin{proposition}
\label{prop:4}
Let $\pi_\xi \,:\, \cat{P}_\xi \,\longrightarrow\, \cat{U}(r,\xi)$ be the algebraic vector bundle defined in \eqref{eq:a43}. Then, the
fiber bundle  $$q\,\,:\,\,\cat{M}'_{\cat{L}}(r,\xi) \,\,\longrightarrow\,\, \cat{U}(r,\xi)$$ defined in   \eqref{eq:a40} is a
$\cat{P}_{\xi}$--torsor over 
$\cat{U}(r,\xi)$.
\end{proposition}

Using the same technique as in the proof of Theorem \ref{thm:1.1}, we can compactify the moduli space $\cat{M}'_{\cat{L}}(r, \xi)$. 
More precisely, we have the following:

\begin{proposition}
\label{prop:5}
There exists a natural algebraic vector bundle $$\Psi \,:\, \cat{F}_\xi \, \,\longrightarrow\,\, \cat{U}(r,\xi)$$  of rank 
$(r^2-1)(\deg{L^*} -g +1) +1$ such that 
$\cat{M}'_{\cat{L}}(r,\xi)$ is embedded in $\p (\cat{F}_\xi)$ with $\p (\cat{F}_\xi) \setminus \cat{M}'_{\cat{L}}(r,\xi)$ being the
hyperplane at infinity.
\end{proposition}

The morphism $q$ in \eqref{eq:a40} induces a morphism of 
Picard groups 
\begin{equation}
\label{eq:a45}
q^* \,:\, \Pic{\cat{U}(r,\xi)} \,\longrightarrow\, \Pic{\cat{M}'_{\cat{L}}(r,\xi)}
\end{equation}
that sends a line bundle $\eta$ over $\cat{U}(r, \xi)$ to a line bundle $q^*\eta$ over 
$\cat{M}'_{\cat{L}}(r,\xi)$. Again imitating  the similar steps as in the proof of the  Theorem \ref{thm:1.2}, we have 
the following:

\begin{proposition}
\label{prop:6}
The homomorphism $q^* \,:\, \Pic{\cat{U}(r,\xi)} \,\longrightarrow\, \Pic{\cat{M}'_{\cat{L}}(r,\xi)}$ in \eqref{eq:a45} 
is an isomorphism.
\end{proposition}

The anchor map $\sharp \,:\, L \,\longrightarrow\, T_X$ induces a morphism 
\begin{equation}\label{eq:a46}
 \alpha \,:\, T^*\cat{U}(r, \xi) \,\longrightarrow\, \cat{P}_\xi
 \end{equation}
of vector bundles over $\cat{U}(r, \xi)$, where $T^*\cat{U}(r, \xi)$ denotes the cotangent bundle of $\cat{U}(r, \xi)$. 
 
Note that, according to \cite[Section 4.6]{FB1}, the dual vector bundle $\cat{P}_\xi^*$ over $\cat{U}(r, \xi)$ admits a Lie algebroid 
structure $(\cat{P}_{\xi}^*,\, [\cdot, \,\cdot],\, \widehat{\sharp})$ such that the dual map
$$\widehat{\sharp}^* \,:\, T^*\cat{U}(r, \xi) \,\longrightarrow\, 
\cat{P}_\xi$$ of $\widehat{\sharp} \,:\, \cat{P}_\xi^* \,\longrightarrow\, T \cat{U}(r,\xi)$ coincides with $\alpha$ in \eqref{eq:a46}.

Let $\Theta$ be the ample generator of the cyclic group
$\Pic{\cat{U}(r, \xi)} \,\cong\, \Z$, where the isomorphism  follows from \cite[Proposition 3.4, (ii)]{R}. Using the fact
that $\cat{P}_\xi^*$ is a Lie algebroid over $\cat{U}(r, \xi)$, we consider the sheaf of $\cat{P}_\xi^*$--connection on $\Theta$.
Recall that a $\cat{P}_\xi^*$--connection on $\Theta$ is a $\C$--linear map 
$$\nabla_{\cat{P}_\xi^*} \,:\, \Theta \,\longrightarrow\, \Theta \otimes \cat{P}_\xi$$
which satisfies the Leibniz rule 
$$\nabla_{\cat{P}_\xi^*} (g t) \,=\, g \nabla_{\cat{P}_\xi^*} (t) + t \otimes d_{\cat{P}_\xi^*}(g),$$
for all local sections $t$ of $\Theta$ and $g$ of 
$\struct{\cat{U}(r, \xi)}$, where $d_{\cat{P}_\xi^*}$ is the following composition of maps
$$\struct{\cat{U}(r, \xi)}\, \xrightarrow{\,\,\,d\,\,\,}\, T^*\cat{U}(r, \xi) \,\xrightarrow{\,\,\,\alpha\,\,\,}\, \cat{P}_\xi.$$

Let $\cat{C}onn_{\cat{P}_\xi^*}(\Theta)$ be the sheaf of all $\cat{P}_\xi^*$-connection on $\Theta$. Then, 
we have a canonical projection 
\begin{equation}\label{eq:21}
\psi \,:\, \cat{C}onn_{\cat{P}_\xi^*}(\Theta) \,\longrightarrow\, \cat{U}(r, \xi).
\end{equation}
Moreover, $\cat{C}onn_{\cat{P}_\xi^*}(\Theta)$ is a quasi-projective variety and it is also a $\cat{P}_\xi$--torsor.

Consider the following standard first order $\cat{P}_\xi^*$--jet exact sequence (also called $\cat{P}_\xi^*$--Atiyah sequence) for the line
bundle $\Theta$ (see \cite{FB} for more details):
\begin{equation}
\label{eq:22}
0 \,\longrightarrow\, \Theta \otimes_{\struct{\cat{U}(r, \xi)}} \cat{P}_\xi\,\longrightarrow\,
 J^1_{\cat{P}^*_\xi}(\Theta) \,\longrightarrow\, \Theta \,\longrightarrow\, 0.
\end{equation}

Applying $\HOM[\struct{\cat{U}(r, \xi)}]{-}{\, \Theta}$ to \eqref{eq:22}, we get
\begin{equation*}
0 \,\longrightarrow\, \END[\struct{\cat{U}(r, \xi)}]{\Theta} \,\xrightarrow{\,\,\,\iota\,\,\,}\, \cat{A}t_{\cat{P}^*_\xi}(\Theta)
\,:=\, \HOM[\struct{\cat{U}(r, \xi)}]{J^1_{\cat{P}^*_\xi}(\Theta)}{\Theta}
\end{equation*}
$$
\xrightarrow{\,\,\,\sigma\,\,\,\,}\, \END[\struct{\cat{U}(r, \xi)}]{\Theta} \otimes \cat{P}^*_\xi \,\longrightarrow\, 0 
$$
which in turn gives the following short exact sequence of vector bundles
\begin{equation}\label{eq:23}
0 \,\longrightarrow\, \struct{\cat{U}(r, \xi)} \,\xrightarrow{\,\,\,\iota\,\,\,}\, \cat{A}t_{\cat{P}^*_\xi}(\Theta)
\,\xrightarrow{\,\,\,\sigma\,\,\,} \cat{P}^*_\xi \,\longrightarrow\, 0 
\end{equation}
over $\cat{U}(r,\xi)$. The vector bundle $\cat{A}t_{\cat{P}^*_\xi}(\Theta)$ in \eqref{eq:23} is called the 
$\cat{P}^*_\xi$--Atiyah bundle associated with $\Theta$, and the morphism $\sigma $ in \eqref{eq:23} is called the symbol operator.

The line bundle $\Theta$ admits a holomorphic $\cat{P}_\xi^*$--connection if and only if the short exact sequence \eqref{eq:23} splits 
holomorphically.

Next we would like to show that the two $\cat{P}_\xi$--torsors 
$\cat{M}_{\cat{L}}'(r, \xi)$  and $\cat{C}onn_{\cat{P}_\xi^*}(\Theta)$ are isomorphic. A few lemmas, which are easy to prove, will
be needed for it.

\begin{lemma}\label{lem:tor1}
Let $Y$ be a smooth projective variety over $\C$.
Let $p_1 \,:\, V \longrightarrow Y$ be a vector bundle and  let $q_1 \,:\, B \,\longrightarrow\, Y$ be a $V$--torsor.
Suppose $\Phi \,:\, V \,\longrightarrow\, W$ be an ${\mathcal O}_Y$--linear morphism between
vector bundles. Then, $V$ acts on $B \times_Y W$ via  
$$v \cdot (b, \,w) \, = \, (b - v,\, w + \Phi_y(v))$$
for every $y \,\in Y$, $v \in V_y$, $b \in B_y$ and $w \in W_y$, and $W$ acts on $B \times_Y W$ via 
$$w' \cdot (b,\, w) \,=\, (b,\, w+w'), $$
for every $y \,\in\, Y$, $b \,\in\, B_y$, $w, \,w' \,\in \,W_y$. The quotient  $(B \times_Y W) / V$ is a $W$--torsor. 
\end{lemma}

Using the notation in Lemma \ref{lem:tor1}, we have the following:

\begin{lemma}
\label{lem:tor2}
Let $B_1 \,\longrightarrow\, Y$ and $B_2 \,\longrightarrow\, Y$ be two $V$--torsors over $Y$ such that $B_1 \,\cong\, B_2$
as $V$--torsors. Then, $(B_1 \times_Y W )/V$ and $(B_2 \times_Y W) / V$ are isomorphic as $W$--torsors.
\end{lemma}

Recall that if we take the Lie algebroid $\cat{L} \,=\, \cat{T}_X$, we get usual holomorphic connections on $E$. Let 
\begin{equation}
\label{eq:moduli-1}
\cat{M}'(r, \xi) 
\end{equation}
denote the moduli space of pairs $(E, \,\mathcal{D})$, where $E$ is a stable vector bundle of rank $r$, 
$\mathcal{D}$ is a holomorphic projective connection on $E$ and $\Lambda ^r E \,\cong\, \xi$. Then, there is a natural projection 
\begin{equation}
\label{eq:proj-1}
\pi'' \,:\, \cat{M}'(r, \xi)\,\longrightarrow\, \cat{U}(r,\xi)
\end{equation}
defined by $(E, \, \mathcal{D})\, \longmapsto\, E$. We note that any stable vector bundle on $X$ admits a holomorphic projective
connection. In fact, every stable vector bundle $E$ of rank $r$ admits a unique holomorphic projective connection whose
monodromy lies in $\text{PU}(r)$ \cite{NS}. 
Note that the moduli space $\cat{M}'(r, \xi) $ is a $T^* \cat{U}(r, \xi)$--torsor.

Let $E$ be a holomorphic vector bundle on $X$ of rank $r$. A holomorphic projective connection on $E$ and a logarithmic connection
on $\bigwedge^r E$ with polar divisor $D$ together define a logarithmic connection on $E$ with polar divisor $D$.

From Lemma \ref{lem:tor1}, the quotient space 
\begin{equation}\label{fib-prod-1}
(\cat{M}'(r, \xi) \times_{\cat{U}(r, \xi)} \cat{P}_\xi)/T^* \cat{U}(r, \xi)
\end{equation}
is a $\cat{P}_\xi$--torsor, where $\cat{P}_\xi$ is in \eqref{eq:a43}. Then, we have:

\begin{lemma}
\label{lem:tor3}
The $\cat{P}_\xi$--torsor $(\cat{M}'(r, \xi) \times_{\cat{U}(r, \xi)} \cat{P}_\xi)/T^* \cat{U}(r, \xi)$ in
\eqref{fib-prod-1} is isomorphic to the $\cat{P}_\xi$--torsor $\cat{M}'_{\cat{L}}(r, \xi)$.
\end{lemma}

Let $\cat{C}onn_{T \cat{U}(r, \xi)}(\Theta)$ be the sheaf of holomorphic connections on the line bundle 
$\Theta$ over $\cat{U}(r, \xi)$. Then, there is a natural projection 
\begin{equation}\label{eq:proj-2}
\psi' \,:\, \cat{C}onn_{T \cat{U}(r, \xi)}(\Theta) \,\longrightarrow\, \cat{U}(r, \xi),
\end{equation}
and $\cat{C}onn_{T \cat{U}(r, \xi)}(\Theta)$ is a $T^*\cat{U}(r, \xi)$--torsor.

Again, in view of maps in \eqref{eq:proj-2} and  \eqref{eq:a43}, we have the quotient of the fiber product 
\begin{equation}\label{eq:fib-prod-2}
(\cat{C}onn_{T \cat{U}(r, \xi)}(\Theta) \times_{\cat{U}(r, \xi)} \cat{P}_\xi)/T^*\cat{U}(r, \xi)
\end{equation}
over $\cat{U}(r, \xi)$. From Lemma \ref{lem:tor1}, this
$(\cat{C}onn_{T \cat{U}(r, \xi)}(\Theta) \times_{\cat{U}(r, \xi)} \cat{P}_\xi)/T^*\cat{U}(r, \xi)$ in
\eqref{eq:fib-prod-2} is a $\cat{P}_\xi$--torsor.

\begin{lemma}\label{lem:tor4}
The $\cat{P}_\xi$--torsor $(\cat{C}onn_{T \cat{U}(r, \xi)}(\Theta) \times_{\cat{U}(r, \xi)} \cat{P}_\xi)/T^*\cat{U}(r, \xi)$ in
\eqref{eq:fib-prod-2} is isomorphic to the $\cat{P}_\xi$--torsor $\cat{C}onn_{\cat{P}_\xi^*}(\Theta)$. 
\end{lemma}

\begin{proposition}\label{prop:tor5}
The moduli spaces $\cat{M}_{\cat{L}}'(r, \xi)$ and $\cat{C}onn_{\cat{P}_\xi^*}(\Theta)$ are isomorphic as 
$\cat{P}_\xi$--torsors.
\end{proposition}

\begin{proof}
The moduli space $\cat{M}'(r, \xi)$ in \eqref{eq:moduli-1} and the space $\cat{C}onn_{T \cat{U}(r, \xi)}(\Theta)$ of 
holomorphic connections on $\Theta$ over $\cat{U}(r, \xi)$, are isomorphic as $T^*\cat{U}(r, \xi)$--torsors
\cite{BR}. Then proposition follows easily from Lemma \ref{lem:tor2}, Lemma \ref{lem:tor3} and Lemma \ref{lem:tor4}.
\end{proof}

The symbol operator 
\begin{equation}
\label{eq:27}
\sigma\,:\, \cat{A}t_{\cat{P}^*_\xi}(\Theta)\,\longrightarrow\, \cat{P}_\xi^*
\end{equation}
as described  in  \eqref{eq:23}, induces a morphism 
\begin{equation}
\label{eq:28}
\cat{S}ym^k(\sigma)\,:\, \cat{S}ym^k \cat{A}t_{\cat{P}^*_\xi}(\Theta) \,\longrightarrow\,
\cat{S}ym^k \cat{P}_\xi^*
\end{equation}
on $k$--th symmetric powers of bundles.
We also have 
\begin{equation}
\label{eq:29}
\cat{S}ym^{k-1} \cat{A}t_{\cat{P}^*_\xi}(\Theta) \,\,\subset\,\, \cat{S}ym^k \cat{A}t_{\cat{P}^*_\xi}(\Theta)
\end{equation}
for all $k \,\geq\, 1$.

In fact, we have $\cat{P}_\xi^*$--symbol exact sequence associated with
$\Theta$ over $\cat{U}(r,\xi)$ (for more details see \cite{FB}),
\begin{equation}
\label{eq:30}
0 \,\longrightarrow\, \cat{S}ym^{k-1}\cat{A}t_{\cat{P}^*_\xi}(\Theta) \,\longrightarrow\, \cat{S}ym^{k} \cat{A}t_{\cat{P}^*_\xi}(\Theta)
\, \xrightarrow{\,\,\cat{S}ym^k (\sigma)\,\,} \cat{S}ym^k \cat{P}_\xi^* \,\longrightarrow\, 0.
 \end{equation}
In other words, we get a filtration 
\begin{equation}
\label{eq:32}
\cat{S}ym^0 \cat{A}t_{\cat{P}^*_\xi}(\Theta) \,\subset\, \cat{S}ym^1 \cat{A}t_{\cat{P}^*_\xi}(\Theta) \,\subset\, \cdots \,\subset 
\, \cat{S}ym^{k-1} \cat{A}t_{\cat{P}^*_\xi}(\Theta)
\, \subset\, \cat{S}ym^k \cat{A}t_{\cat{P}^*_\xi}(\Theta) \,\subset\, 
\cdots
\end{equation}
such that 
 \begin{equation}
 \label{eq:33}
 \cat{S}ym^{k} \cat{A}t_{\cat{P}^*_\xi}(\Theta) / \cat{S}ym^{k-1} \cat{A}t_{\cat{P}^*_\xi}(\Theta)
\,\, \cong\,\, \cat{S}ym^k \cat{P}_\xi^*
 \end{equation}
for all $k\, \geq \,1$. We have the following commutative diagram
\begin{equation}
\label{eq:cd3}
\xymatrix@C=2em{
0 \ar[r] & \cat{S}ym^{k-1}\cat{A}t_{\cat{P}^*_\xi}(\Theta) \ar[d] \ar[r] & \cat{S}ym
^k\cat{A}t_{\cat{P}^*_\xi}(\Theta) 
\ar[d] \ar[r]^{\sigma_k} &  \cat{S}ym^k  \cat{P}^*_\xi \ar[d] \ar[r] & 0 \\
0 \ar[r] &  \cat{S}ym^{k-1} \cat{P}^*_\xi \ar[r] & 
\frac{\cat{S}ym^k \cat{A}t_{\cat{P}^*_\xi}(\Theta)}{\cat{S}ym^{k-2} \cat{A}t_{\cat{P}^*_\xi}(\Theta)} \ar[r] &  \cat{S}ym^k  \cat{P}^*_\xi \ar[r] & 0 
}
\end{equation}
which gives the following commutative diagram of long exact sequences of cohomologies
\begin{equation}\label{eq:cd4}
\xymatrix@C=2em{
\cdots \ar[r] & \coh{0}{\cat{U}(r,\xi)}{\cat{S}ym^k \cat{P}_\xi^*} \ar[d] \ar[r]^{\delta'_{k}} & \coh{1}
{\cat{U}(r,\xi)}{\cat{S}ym^{k-1} \cat{A}t_{\cat{P}^*_\xi}(\Theta)} \ar[d] 
\ar[r] & \cdots  \\
\cdots \ar[r] & \coh{0}{\cat{U}(r,\xi)}{\cat{S}ym^k 
\cat{P}^*_\xi}       \ar[r]^{\delta_{k}} & \coh{1}
{\cat{U}(r,\xi)}{ \cat{S}ym^{k-1} \cat{P}^*_\xi} 
\ar[r] & \cdots }
\end{equation}

{}From Proposition \ref{prop:tor5} it follows that there are some natural globally defined functions
on $\cat{C}onn_{\cat{P}_\xi^*}(\Theta)$. To see this, take any point $x_i\, \in\, S$ (see \eqref{an1});
recall that $L\,=\, T_X(- \log S)$. Take any integer $2\, \leq\, k\, \leq\, r$. Sending any logarithmic
connection $(E,\, \mathcal{D})\, \in\, \cat{M}_{\cat{L}}'(r, \xi)$ on $X$ to ${\rm Trace}({\rm Res}({\mathcal D},\, x_i)^k)
\, \in\, {\mathbb C}$, where ${\rm Res}({\mathcal D},\, x_i)\, \in\, {\rm End}(E_{x_i})$ is the residue of
${\mathcal D}$ at $x_i$ (see \cite{D}) we get an algebraic function on $\cat{M}_{\cat{L}}'(r, \xi)$. Now using the
isomorphism of $\cat{M}_{\cat{L}}'(r, \xi)$ with $\cat{C}onn_{\cat{P}_\xi^*}(\Theta)$ in
Proposition \ref{prop:tor5} this produces an algebraic function on $\cat{C}onn_{\cat{P}_\xi^*}(\Theta)$.
This algebraic function on $\cat{C}onn_{\cat{P}_\xi^*}(\Theta)$ will be denoted by $\Psi_{i,k}$.
We have the function
\begin{equation}\label{an2}
\Psi\,\,:=\,\, \bigoplus_{i=1}^m \bigoplus_{k=2}^r \,\,:\,\, \cat{C}onn_{\cat{P}_\xi^*}(\Theta)
\,\, \longrightarrow\,\, {\mathbb C}^{m(r-1)}.
\end{equation}

Let $$Z\, \subset\, \cat{C}onn_{\cat{P}_\xi^*}(\Theta)$$ be a general fiber of the map $\Psi$ in \eqref{an2}.
So we have
\begin{equation}\label{an3}
T\psi(Z)\, \subset\, T\cat{U}(r, \xi)\big\vert_{\psi(Z)},
\end{equation}
where $\psi$ is the projection in \eqref{eq:21}.

Consider the homomorphism $\alpha$ in \eqref{eq:a46}. Let
$$
\alpha^*  \,:\, \cat{P}^*_\xi \,\longrightarrow\, T\cat{U}(r, \xi)
$$
be the dual homomorphism. Define
\begin{equation}\label{an4}
\cat{P}^*_Z \,:=\, (\alpha^*)^{-1}(T\psi(Z)) \, \longrightarrow\, Z
\end{equation}
(see \eqref{an3}). This, equipped with the restriction of $\alpha^*$, is a Lie algebroid on $Z$.

Just as in the bottom row of \eqref{eq:cd3}, we have a short exact sequence
\begin{equation}\label{an5}
0\, \longrightarrow\, \cat{S}ym^{k-1} \cat{P}^*_Z \, \longrightarrow\, {\mathcal W}
\, \longrightarrow\, \cat{S}ym^{k} \cat{P}^*_Z \, \longrightarrow\, 0.
\end{equation}
As in the bottom row of \eqref{eq:cd4}, we have
\begin{equation}\label{an6}
\delta_{k,Z}\,\, :\,\,  \coh{0}{\cat{U}(r,\xi)}{\, \cat{S}ym^k 
\cat{P}^*_Z} \, \, \longrightarrow\,\, \coh{1}{\cat{U}(r,\xi)}{\, \cat{S}ym^{k-1} \cat{P}^*_Z}.
\end{equation} 

We will prove that the homomorphism $\delta_{k,Z}$ in \eqref{an6} is injective. But this follows by imitating
the argument in \cite{BR}. Now, as in \cite{BR}, from \eqref{an6} we have the following:

\begin{proposition}\label{prop1}
There are no non-constant algebraic function on $Z$.
\end{proposition}

The following is an immediate consequence of Proposition \ref{prop1}.

\begin{theorem}\label{thm1}
For any algebraic function $F$ on $\cat{C}onn_{\cat{P}_\xi^*}(\Theta)$, there is an algebraic function
$$
\widetilde{F} \, :\, {\mathbb C}^{m(r-1)} \, \longrightarrow\, {\mathbb C}
$$
such that $F\,=\, \widetilde{F}\circ \Psi$, where $\Psi$ is the map in \eqref{an2}.
\end{theorem}

\begin{remark}\label{rem1}
In view of Proposition \ref{prop:tor5}, the space of all algebraic functions on $\cat{M}_{\cat{L}}'(r, \xi)$
are described by Theorem \ref{thm1}.
\end{remark}

\section{Rationally connectedness and Divisor at infinity}\label{Div}

In \cite{KS}, rationality of moduli spaces of vector bundles over a smooth projective curve has been studied. Also, the rationality and 
the rationally connectedness of the moduli space of logarithmic connections has been discussed in \cite{AS3}.

Motivated by this, there is a natural question whether the moduli spaces $\cat{M}_{\cat{L}}(r,d)$ and $\cat{M}_{\cat{L}}(r,\xi)$ are rational. 
For the theory of 
rational varieties, we refer \cite{Ko}. 
In this section, we show that the moduli space $\cat{M}_{\cat{L}}(r,d)$ is not rational and the moduli space 
$\cat{M}_{\cat{L}}(r,\xi)$ is rationally connected.

Recall that   a smooth complex variety $V$  
is said to be rationally connected if any two general points on $V$
can be connected by a rational curve in $V$. A Rational variety is always rationally connected. But the converse is not true.  The following lemma is an easy consequence of  the definition.

\begin{lemma}
\label{lem:6.1}
Let $f \,:\, \cat{Y} \,\longrightarrow\, \cat{X}$ be a dominant rational map of complex 
algebraic varieties with $\cat{Y}$ rationally connected. Then,  $\cat{X}$ is rationally connected.
\end{lemma}

\begin{theorem}[{\cite[Theorem 1.1]{KS}}] 
The moduli space 
$\cat{U}(r,d)$ is birational to $J(X) \times \A^{(n^2-1)(g-1)}$, where $J(X)$ is the Jacobian of $X$.
\end{theorem}

Note that $J(X)$ is not rationally connected,
because it does not contain any rational curve. 
Therefore, $\cat{U}(r,d)$ is not rationally connected.

\begin{proposition}
\label{thm:6.3}
The moduli space $\cat{M}_{\cat{L}}(r,d)$ is not rational.
\end{proposition}

\begin{proof}
We show that the moduli space $\cat{M}_{\cat{L}}(r,d)$ is not rationally connected. Let 
$$p\,:\, \cat{M}_{\cat{L}}'(r,d) \,\longrightarrow\, \cat{U}(r,d)$$
be the morphism of varieties defined in \eqref{eq:7}.
Suppose that $\cat{M}'_{\cat{L}}(r,d)$ is rationally connected. Then, from Lemma \ref{lem:6.1}, 
$\cat{U}(r,d)$ is rationally connected, which is not true. 
Since $\cat{M}'_{lc}(n,d)$ is an open dense subset of 
$\cat{M}_{lc}(n,d)$, it follows that $\cat{M}_{lc}(n,d)$ is not rational.
\end{proof}

\begin{lemma}[{\cite[Corollary 1.3]{GHS}}]
\label{lem:6.4}
Let $f\,:\, \cat{X} \,\longrightarrow\, \cat{Y}$ be any dominant morphism 
of complex varieties. If $\cat{Y}$ and the general 
fiber of $f$ are rationally connected, then $\cat{X}$
is rationally connected.
\end{lemma}

\begin{proposition}
\label{prop:6.5}
The moduli space $\cat{M}_{\cat{L}}(r,\xi)$ is rationally connected.
\end{proposition}

\begin{proof}
Consider the dominant  morphism 
$$q\,:\, \cat{M}'_{\cat{L}}(r,\xi)\, \longrightarrow\, \cat{U}(r,\xi)$$
defined in \eqref{eq:a40}. As observed earlier, every 
fiber of $q$ is an affine space and hence it is rationally 
connected. Since $\cat{U}(r,\xi)$ is rationally connected, $\cat{M}'_{\cat{L}}(r,\xi)$ is 
rationally connected, which follows from Lemma \ref{lem:6.4}.
Now rationally connectedness is a birational invariant, and $\cat{M}'_{\cat{L}}(r,\xi)$ is a dense open subset of 
$\cat{M}_{\cat{L}}(r,\xi)$.
\end{proof}

Therefore, we have  a natural question.
\begin{question}
\label{q:1}
Is the moduli space $\cat{M}_{\cat{L}}(r,\xi)$ rational?
\end{question}

Next, we discuss the numerical effectiveness of the divisor at infinity. In view of Proposition \ref{prop:5}, we have a compactification 
$\p (\cat{F}_\xi)$ of the moduli space $\cat{M}'_{\cat{L}}(r,\xi)$ such that the complement ${\bf H_0} \,:=\, \p (\cat{F}_\xi) \setminus 
\cat{M}'_{\cat{L}}(r,\xi)$ is the hyperplane at infinity, and we call ${\bf H_0}$ the divisor at infinity.

\begin{proposition}
\label{prop:1.5}
Let ${\bf H_0} \,:=\, \p (\cat{F}_\xi) \setminus \cat{M}'_{\cat{L}}(r,\xi)$ be the smooth divisor.
Then,
the divisor ${\bf H_0}$ is numerically effective if and only if the vector bundle bundle $\cat{P}^*_\xi$ over the moduli
space $\cat{U}(r,\xi)$ is numerically effective. 
\end{proposition}

\begin{proof}\label{proof_prop_1.5}
Recall that  the vector bundle $\cat{P}^*_\xi$ is 
numerically effective if and only if the tautological 
line bundle $\struct{\p (\cat{P}^*_\xi)}(1)$ is 
numerically effective.
Let $$\Psi \,:\,\cat{F}_\xi \,\longrightarrow \,\cat{U}(r, \xi)$$ be the vector bundle  in Proposition \ref{prop:5}. Then,
the  divisor ${\bf H_0}$ is numerically effective if and only if the restriction of the line bundle
$\struct{\p (\cat{F}_\xi)}({\bf H_0})$ to ${\bf H_0}$ is numerically effective.

First we show that the divisor $\bf{H}_0$ is canonically 
isomorphic to projective bundle $\p ( \cat{P}^*_\xi)$.
Let 
\begin{equation*}
\label{eq:a61}
\widetilde{\Psi} \,:\, \p (\cat{F}_\xi) \,\longrightarrow\, \cat{U}(r,\xi)
\end{equation*}
be the projective bundle. Let $E \,\in\, \cat{U}(r,\xi)$, and 
\begin{equation}
\label{eq:a62}
\theta \,\in\, \widetilde{\Psi}^{-1}(E) \cap \bf{H}_0\, \subset\, 
\p (\cat{F}_\xi).
\end{equation}
Then, $\theta$ represents a hyperplane in the fiber 
$\cat{F}_\xi(E) \,=\, q^{-1}(E)^{\vee}$ of the vector bundle $\cat{F}_\xi$, where $q$ is defined in 
\eqref{eq:a40}. Let $H_{\theta}$ denote this hyperplane represented by $\theta$.
Note that $H_{\theta} \,\subset\, q^{-1}(E)^{\vee}$ and $q^{-1}(E)$ is the affine space modelled 
over the vector space $\coh{0}{X}{\ad{E} \otimes L^*}$
which is a fiber of the vector bundle $\cat{P}_\xi$ over $E$.
Therefore
for any $f \,\in\, H_{\theta}$, that is, for an affine linear map $f \,:\, q^{-1}(E) \,\longrightarrow\, \C$, we have 
$$\text{d}f \in (\cat{P}_\xi(E))^*\, = \,\cat{P}^*_\xi(E).$$
Since $\theta \,\in\, \bf{H}_0$, and $H_\theta$ is a hyperplane, the subspace of $\cat{P}^*_\xi(E)$
generated by $\{\text{d}f\}_{f \in H_{\theta}}$ is a
hyperplane in  $ \cat{P}^*_\xi(E)$. Let $\widetilde{\theta} \,\in\, \p ( \cat{P}^*_\xi)$ denote this hyperplane.
Thus, we get a canonical isomorphism 
\begin{equation}
\label{eq:a63}
{\bf H}_0 \,\cong\, \p (\cat{P}^*_\xi)
\end{equation}
by sending $\theta$ to $\widetilde{\theta}$.

Let $\cat{N}_{\p(\cat{F}_\xi)/{\bf H_0}}$ denote the normal 
bundle of the divisor ${\bf H_0} \,\subset \,\p (\cat{F}_\xi)$.
Note that from Poincar\'e adjunction formula we have the following:
\begin{equation}
\label{eq:a60}
\struct{\p (\cat{F}_\xi)}({\bf H_0}) \vert_{\bf H_0} \,\cong\, \cat{N}_{\p(\cat{F}_\xi)/{\bf H_0}}.
\end{equation}
So, ${\bf H_0}$ is numerically effective if and only if the normal bundle $\cat{N}_{\p(\cat{F}_\xi)/{\bf H_0}}$
is numerically effective. 

Thus, to prove the proposition it is enough to show that 
the normal bundle $\cat{N}_{\p(\cat{F}_\xi)/\bf{H}_0}$ is canonically isomorphic to the tautological 
line bundle $\struct{\p ( \cat{P}^*_\xi)}(1)$. 
 
Next, note that the fiber $\cat{N}_{\p(\cat{F}_\xi)/{\bf H_0}}(\theta)$ of the normal bundle $\cat{N}_{\p(\cat{F}_\xi)/ {\bf H_0}}$
is canonically isomorphic to the quotient $\cat{F}_\xi(E) / H_{\theta}$. 
Consider the morphism 
$$\cat{F}_\xi(E) \,\longrightarrow\, \cat{P}^*_\xi(E)$$
between vector spaces defined by $f \,\longmapsto\, \text{d} f$.
Since the image of the hyperplane $H_{\theta}$ is contained in $\widetilde{\theta}$, we have a well defined morphism on quotients 
$$\cat{F}_\xi(E)/ H_{\theta} \,\longrightarrow\, \cat{P}^*_\xi(E)/ \widetilde{\theta},$$
which is an isomorphism. Recall that the fiber of the tautological line bundle 
$\struct{\p (\cat{P}^*_\xi)}(1)$ at $E$ is canonically 
identified with the quotient $\cat{P}^*_\xi(E)/ \widetilde{\theta}$.
This completes the proof.
\end{proof}

A similar result can be proved for the smooth divisor on the compactification of the moduli space $\cat{M}_{\cat{L}}'(r,d)$.  Let ${\bf 
H} \,:= \,\p ( \cat{F}) \setminus \cat{M}'_{\cat{L}}(r,d)$ be the smooth divisor at infinity, where $\p (\cat{F})$ is the compactification of 
the moduli space $\cat{M}'_{\cat{L}}(r,d)$ as in Theorem \ref{thm:1.1}. Then using the same steps as in proof of Proposition 
\ref{prop:1.5}, we can show the following:

\begin{proposition}
\label{prop:8}
The smooth divisor $\bf{H}$ is numerically effective 
if and only if the vector bundle $ \cat{P}^*$ is numerically effective, where $\cat{P}^*$ is the dual of the vector bundle defined in \eqref{eq:6}.
\end{proposition}

\section*{Acknowledgements} 

The second author is thankful  to  David Alfaya for going through an earlier version of the manuscript. 
The first
author is partially supported by a J. C. Bose Fellowship (JBR/2023/000003). The second author is partially supported by SERB SRG Grant SRG/2023/001006.


\begin{thebibliography} {999}

\bibitem{AO} D.~Alfaya, A.~Oliveira,  Lie algebroid connections, twisted Higgs bundles and motives of moduli spaces, arXiv:2102.12246v4.

\bibitem{A} M.~F.~Atiyah, Complex analytic connections in 
fibre bundles, {\it Trans.~Amer.~Math.~Soc.} {\bf 85}(1957), 181--207.

\bibitem{BNR} A.~Beauville, M.~S.~Narasimhan and S.~Ramanan, Spectral curves and the generalised theta divisor, {\it J. Reine Angew. Math.} 
{\bf 398} (1989), 169--179.

\bibitem{BGL}
I.~Biswas, P.~B.~Gothen, and M.~Logares. On moduli spaces of hitchin pairs, {\it Math. Proc.
Cambridge Phil. Soc.} {\bf 151} (2011), 441--457.

\bibitem{BR} I.~Biswas and N.~Raghavendra,  Line bundles over a 
moduli space 
of logarithmic connections on a Riemann surface. {\it Geom. Funct. 
Anal}, {\bf 15} (2005), 780--808.

\bibitem{BM} A.~Borel , P.-P.~Grivel, B.~Kaup, 
A.~Haefliger, B.~Malgrange and F.~Ehlers, {\it 
Algebraic D-modules}. Perspectives in Mathematics, 2.
Academic Press, Inc., Boston, MA, 1987.

\bibitem{FB1}
F.~Bottacin, Symplectic geometry on moduli spaces of stable pairs, {\it Ann. Sci. \'Ecole Norm. Sup.}
{\bf 28} (1995), 391--433.

\bibitem{FB} F.~Bottacin, Atiyah classes of Lie algebroids. {\it Current trends in analysis and its applications}, 375--393, Trends Math., 
Birkhäuser/Springer, Cham, 2015.

\bibitem{BOM}
M.~N.~
Boyom, KV-cohomology of Koszul–Vinberg algebroids and Poisson manifolds, {\it Int. Jour. Math.} {\bf 16} (2005), 1033--1061.

\bibitem{CS}
S.~Chemla, A duality property for complex Lie algebroids, {\it Math. Z.} {\bf 232} (1999), 367--388.

\bibitem{D} P.~Deligne, {\it Equations diff\'erentielles \'a points 
singuliers r\'eguliers}. Lecture Notes in Mathematics, vol. 163.
Springer, Berlin(1970).

\bibitem{DN} J.~M.~Drezet and M.~S.~Narasimhan, Groupe de Picard des vari\'et\'es de modules de fibr\'es semi-stables sur les courbes 
alg\'ebriques, {\it Invent. Math.} {\bf 97} (1989), 53--94.

\bibitem{ELW} S.~Evens and J.-H.~Lu, and A.~Weinstein, Transverse measures, the modular class and a cohomology pairing for Lie algebroids, 
{\it Quart. Jour. Math.} {\bf 50} (1999), 417--436.

\bibitem{RF} R.~L.~Fernandes,
Lie algebroids, holonomy and characteristic classes,
{\it Adv. Math.} {\bf 170} (2002), 119--179.

\bibitem{GHS}
T.~Graber, J.~Harris and J.~Starr, Families of rationally connected varieties, {\it J. Amer. Math.
Soc.} {\bf 16} (2003), 57--67.

\bibitem{KS} 
A.~King and  A.~Schofield, {Rationality of moduli space of 
vector bundles on curves}, {\it Indag. Math. (N.S.)}
{\bf 10} (1999), 519--535.

\bibitem{Ko} 
J.~Koll\'ar, {\it Rational curves on algebraic varieties}, Ergebnisse der Mathematik und ihrer
Grenzgebiete, 3 Folge, Band 32, Springer Verlag, 1996.

\bibitem{LK1} L. Krizka, Moduli spaces of lie algebroid connections, {\it Archivum Mathematicum,} {\bf 44} (2009), 403--418.

\bibitem{LK2} L. Krizka, Moduli spaces of flat lie algebroid connections. arXiv:1012.3180, 2010.

\bibitem{KM}
K.~Mackenzie, Lie algebroids and Lie pseudoalgebras, {\it Bull. London Math. Soc.} {\bf 27} (1995),
97--147.

\bibitem{M} M.~Maruyama, Openness of a family of torsion free
sheaves, {\it Jour. Math. Kyoto Univ.} {\bf 16} (1976), 627--637.

\bibitem{NS} M. S. Narasimhan and C. S. Seshadri, Stable and unitary vector
bundles on a compact {R}iemann surface, {\em Ann. of Math.} {\bf 82} (1965), 540--567.

\bibitem{N}
N.~Nitsure, Moduli space of semistable pairs on a curve, 
{\it Proc. London Math. Soc.}, {\bf 62} (1991), 275--300.

\bibitem{P1} J.~Pradines, Th\'eorie de Lie pour les groupo\"ides diff\'erentiables. Relations entre propri\'et\'es locales et globales.  
{\it C.~R.~Acad.~Sci.~Paris} {\bf 263} (1966), 907–-910.

\bibitem{P2}
J.~Pradines, Th\'eorie de Lie pour les groupo\"ides 
diff\'erentiables. Calcul diff\'erenetiel dans la 
cat\'egorie des groupo\"ides infinit\'esimaux.  
{\it C.~R.~Acad.~Sci.~Paris} {\bf 264} (1967), 245–-248.

\bibitem{R} S.~Ramanan, The moduli space of vector bundles over an algebraic curve,
{\it Math. Ann.} {\bf 200} (1973), 69–-84.

\bibitem{S1} C.~T.~Simpson, Moduli of representations of fundamental group of a smooth projective variety, \textbf{I}, {\it Inst. 
Hautes \'Etudes Sci. Publ. Math.} {\bf 79} (1994), 47--129.

\bibitem {S2} C.~T.~Simpson, {Moduli of representations of fundamental group of a smooth projective variety,} \textbf{II}, 
{\it Inst. Hautes \'Etudes Sci. Publ. Math.} {\bf 80} (1994), 5--79. 

\bibitem{AS1} A.~Singh, Moduli space of logarithmic connections singular over a finite subset of comapct Riemann surface.  {\it Math. 
Res. Lett.} {\bf 28} (2021), 863-–887.

\bibitem{AS2} A.~Singh, Moduli space of rank one logarithmic connections over a compact Riemann surface. {\it C. R. Math. Acad. Sci. 
Paris} {\bf 358} (2020), 297-–301.

\bibitem{AS3}A.~Singh, A note on the moduli spaces of holomorphic and logarithmic connections over a compact Riemann surface. {\it Ann. 
Global Anal. Geom.} {\bf 62} (2022), 579–-601.

\bibitem{T1} P.~Tortella, $\Lambda$-modules and holomorphic Lie algebroids. PhD thesis, Scuola Internazionale Superiore di Studi 
Avanzati, 2011.
 
\bibitem{T2}P.~Tortella, $\Lambda$-modules and holomorphic Lie algebroid connections. {\it Cent. Eur. J. Math.} {\bf 10} (2012),
1422–-1441.

\bibitem{W} A.~Weil, G\'en\'eralisation des fonctions ab\'eliennes, {\it J.~Math.~Pures
Appl.} \textbf{17} (1938), 47--87.

\bibitem{WE}
A.~Weinstein,  The integration problem for complex Lie algebroids, 93–109. In Geometry
to Quantum Mechanics, Progress in Mathematics Series, vol. 252. Boston, MA: Birkh\"auser
Boston, 2007

\end{thebibliography}
\end{document}